\documentclass[12pt,a4paper]{article}
\usepackage{amssymb}

\usepackage{graphicx}
\usepackage{amsmath}
\usepackage{amsmath,amssymb}%

\usepackage{amsfonts}
\usepackage{amsthm,amscd}
\usepackage[english]{babel}
\usepackage{graphicx}
\usepackage{amsmath}
\usepackage{amssymb}
\usepackage{amstext}

\newtheorem{theorem}{Theorem}
\newtheorem{corollary}[theorem]{Corollary}
\newtheorem{definition}[theorem]{Definition}

\newtheorem{lemma}[theorem]{Lemma}

\newtheorem{proposition}[theorem]{Proposition}

\begin{document}

\title{Entropy, Pressure and Duality for Gibbs plans in  Ergodic Transport}

\author{A. O. Lopes, J. K. Mengue, J. Mohr and  R. R. Souza} 
\date{\today}

\maketitle

\centerline{Instituto de Matem\'atica, UFRGS - Porto Alegre, Brasil}
\bigskip

Let $X$ be a finite set and $\Omega=\{1,...,d\}^{\mathbb{N}}$ be the Bernoulli space. Denote by $\sigma$ the shift map acting on $\Omega$.
We consider a fixed Lipschitz cost (or potential) function $c: X \times \Omega \to \mathbb{R}$ and an associated Ruelle operator. We introduce the concept of Gibbs plan for $c$, which is a probability on $X \times \Omega$ such that the $y$ marginal is invariant. Moreover, we define entropy, pressure and equilibrium plans. The study of equilibrium plans can be seen as a generalization of the equilibrium probability problem where the concept of entropy for plans is introduced.  We show that an equilibrium plan is a Gibbs plan.

For a fixed probability $\mu$ on $X$ with supp($\mu$)$=X$, define $\Pi(\mu,\sigma)$ as the set of all Borel probabilities $\pi \in P(X\times \Omega)$ such that the $x$-marginal of $\pi$ is $\mu $ and the $y$-marginal of $\pi$ is $\sigma$-invariant.
We also investigate the pressure problem over $\Pi(\mu,\sigma)$, that is with constraint $\mu$.
Our main result is a Kantorovich duality Theorem on this setting. The pressure without constraint plays an important role in the establishment of the notion  of
admissible pair. Basically we want to transform a problem of pressure with a constraint $\mu$ on $X$ in a problem of pressure without constraint. Finally, given a parameter $\beta$, which  plays the role of the inverse of
temperature, we consider equilibrium plans for $\beta c$ and its limit $\pi_\infty$, when $\beta \to \infty$, which is also known as ground state. We compare this with other previous results on Ergodic  Transport at temperature zero.



\section{Introduction}

Kantorovich duality is a general theoretical tool for solving problems. The practical problems where one can get explicit solutions in Classical Transport Theory are in general obtained via duality techniques and the complementary slackness condition. Here we investigate this kind of result in a dynamical setting associated to a generalization of
Thermodynamic Formalism which fits well the Transport setting.

We want to show that:

1) the principle of maximizing pressure in Thermodynamic Formalism corresponds in the more general dynamical setting to  Kantorovich Duality (section 3).

2) the slackness condition is given by a simple equation which uses a generalized Ruelle operator (presented in section 2).

\medskip

Let $X$ be a finite set and $\Omega=\{1,...,d\}^{\mathbb{N}}$ the Bernoulli space with the usual metric\footnote{ $d(z,y) = \frac{1}{2^{n}}$ when $z=(z_0,z_1,...),\, y=(y_0,y_1,...)$, $z_n\neq y_n$ and $z_j=y_j,\, j<n.$}.
 We denote by  $\sigma$ the shift map acting on $\Omega$, by $P(X)$ the set of probabilities over $X$ and by $P(\Omega)$ the set of probabilities acting on the Borel sigma-algebra $\mathcal{B}(\Omega)$. Let $C(X)$ be the set of functions from $X$ to $\mathbb{R}$ and $C(\Omega)$ be the set of continuous functions from $\Omega$ to $\mathbb{R}$. We  denote by $(x,y)$ the variables on the space $X\times \Omega$. 

A Borel probability $\pi $ on $X\times \Omega$ is called a plan. For a fixed $\mu \in P(X)$ such that supp($\mu$)$=X$, define $\Pi(\mu,\sigma)$ as the set of all plans satisfying
\begin{equation}
\left\{ \begin{array}{l}
\int_{X\times \Omega} f(x) \, d\pi(x,y) = \int_X f(x) \, d\mu(x) \ \ \text{for any} \, f\in C(X), \\
\int_{X\times \Omega}  g(y) \,d\pi(x,y) = \int_{X\times \Omega}  g(\sigma(y)) \, d\pi(x,y) \ \ \text{for any} \, g\in C(\Omega),
\end{array}
\right.
\label{marginal}
\end{equation}
which means, the set of probabilities $\pi$ such that the $x$-marginal of $\pi$ is the fixed probability $\mu \in P(X)$ and the $y$-marginal of $\pi$ is $\sigma$-invariant.
Define $\Pi(\cdot,\sigma)$ as the set of plans such that its $y$-projection is $\sigma$-invariant.

We will introduce the entropy $H(\pi)$ of a plan $\pi \in \Pi(\cdot,\sigma)$  and the pressure $P(c)$ of a Lipschitz cost $c:X \times \Omega \to \mathbb{R}$. More precisely

$$P(c)= \sup_{\pi \in \Pi(\cdot,\sigma)} \int_{X\times \Omega}  c  \, d\pi\, +\, H(\pi).$$

We will show in section 3 the following result which is the  {\bf Kantorovich duality} in the (Transport) Thermodynamic Formalism setting:

\vspace{0.5cm}

\noindent
\textbf{Variational Principle}
\[\inf_{\varphi:P(c-\varphi)= 0} \int_{X} \varphi(x) \, d\mu = \sup_{\pi \in \Pi(\mu,\sigma)} \int_{X\times \Omega}  c  \, d\pi\, +\, H(\pi).\]
\textit{The infimum and supremum will be attained by unique elements $\tilde\varphi$ and $\tilde\pi$.}
\vspace{0.5cm}

If $X$ has only one element we get the classical Thermodynamical Formalism setting. The function $\tilde\varphi$ plays in some sense the role of the main  eigenvalue of the transfer operator. If $X$ has only one element, $\tilde\varphi$ coincides with $\log(\lambda)$, where $\lambda$ is the main eigenvalue of the classical Ruelle Operator \cite{PP}.

When $X$ has more than one point the main issue in the theory is to be able to characterize the optimal $\tilde\pi$. The main point here is to transform a problem of pressure with a constraint $\mu$, namely
$$ \sup_{\pi \in \Pi(\mu,\sigma)} \int_{X\times \Omega}  c  \, d\pi\, +\, H(\pi),$$
in a problem of pressure  without constraint
$$ P(\tilde{c})=\sup_{\pi \in \Pi(\cdot,\sigma)} \int_{X\times \Omega}  \tilde{c}  \, d\pi\, +\, H(\pi).$$
The $\tilde{\varphi}$ helps to do that because $\tilde{c}= c -\tilde{\varphi}$. This can be achieved via  the duality of Theorem \ref{dualidade}.
We will present a worked example (see example 3 on page 18) to illustrate how  one can get explicit solutions  of the above mentioned problem.

\bigskip






Section 2 generalize to plans what is known in Thermodynamic Formalism. Most of the proofs are a kind of standard generalization of the classical setting.
In section 2,  $\mu$ is not fixed and the results are  about plans and not exactly about transport (optimal plans with a fixed $\mu$ as marginal). We need this part in section 3 where $\mu$ is fixed and our main result is proved.

\medskip

Now we will present some motivations from results about Ergodic Optimization and transport contained in \cite{LM}.

Consider a Lipschitz function $c:X \times \Omega \to \mathbb{R}$. We denote by $\pi_{\text{opt}(c)}$ any optimal plan in $\Pi(\mu,\sigma)$, which means  that $\pi_{\text{opt}(c)}$ satisfies
\[\sup_{\pi \in \Pi(\mu,\sigma)}\int_{X\times \Omega}  c\, d\pi = \int_{X\times \Omega}  c \, d\pi_{\text{opt}(c)}.\]
It is well known that the optimal plan $\pi_{\text{opt}(c)}$ may not be unique. Associated  to the problem of finding an optimal plan $\pi_{\text{opt}(c)}$ and determining $\displaystyle\sup_{\pi \in \Pi(\mu,\sigma)}\int_{X\times \Omega}  c\, d\pi$ it is natural to consider the dual problem: if $c$ is Lipschitz continuous, there exist Lipschitz functions $V: \Omega\to \mathbb{R}$ and $m: X \to \mathbb{R}$ such that
\begin{equation}\label{subacao}
c(x,y) + V(y) - V (\sigma(y)) - m(x) \leq 0\,,
\end{equation}
and
\begin{equation}\label{intigual}
\int_{X\times \Omega}  c(x,y) \, d\pi_{\text{opt}(c)} = \int_{X}  m(x) \, d\mu.
\end{equation}

This result was proved in \cite{LM}. Results of such nature are part of what is called Ergodic Transport Theory. We point out that in \cite{LM} it is considered a minimization cost problem and here we consider a maximization cost problem. There is no conceptual difference in both settings.

If $V$ satisfies (\ref{subacao}) and \eqref{intigual}, we say that $V$ is a subaction associated to $c$ and $\mu$.
We say that $V$ is a calibrated subaction if for each given $y\in \Omega$ there exist $x\in X$ and a pre-image $w$ of $y$, such that
\[c(x,w) +V(w) - V(y) - m(x) = 0.\]

The problem above is very much related with the questions which are usually considered in Ergodic Optimization. Indeed, if we consider the particular case where $X$ has a unique point $x$, then $\mu$ will be the Dirac measure $\delta_x$, and any plan in $\Pi(\delta_x,\sigma)$ is a direct product $\delta_x \times \nu$, where $\nu$ is a invariant measure. In this case, we can identify $X\times \Omega$ with $\Omega$, $\Pi(\delta_x,\sigma)$ with the set of invariant measures on $\Omega$, $c(x,y)$ with a potential $A(y)$, and in this case $m$ will be constant and  equal  to the number
\[ m(A)=\sup_{\nu\, \text{-invariant}} \int_{ \Omega}  A \, d\nu,\]
which is called sometimes the maximal value of $A$.
In the Ergodic Transport setting we have that $m$ is a function on $x$ and also get the validity of the equations (\ref{subacao}) and (\ref{intigual}).


When $X$ has two or more points the function $m(x)$ is strongly related with the initial fixed $\mu$ in the following sense:
%
%
%
 If $c$ and $m$ satisfy \eqref{subacao} and \eqref{intigual} for some $V$, then
\[\sup_{P\in \Pi(\cdot,\sigma)} \int_{X\times \Omega} c(x,y) -m(x) \, dP = \sup_{\pi\in \Pi(\mu,\sigma)} \int_{X\times \Omega} c(x,y) -m(x) \, d\pi = 0.\]

It is well known that Ergodic Optimization is related with Thermodynamic Formalism via the zero temperature limit (see \cite{BLLbook}). In this way it is natural to investigate the possible generalizations of the transfer operator (also called Ruelle operator), and other properties which appear in Thermodynamic Formalism for the Ergodic Transport setting. This will be done in sections 2 and 3. 





We will show in section \ref{zero} that in the zero temperature limit the function $\tilde\varphi(x)$  will correspond to the function $m(x)$ previously defined. The optimal plan $\pi_{opt(c)}$ will correspond, in the zero temperature, to the limit of the equilibrium plans $\tilde\pi$.

The analogous questions in the case where $X$ is not finite and $\mu$ is a general probability on $X$ will require
a different type of transfer operator. This will introduce some technical difficulties which are similar to the ones analyzed in \cite{LMMS},
where it is considered a general a priory probability. We will not address here this more general problem.

In the appendix we will present some technical results which are needed in our reasoning.

The setting presented here is different from \cite{GL2} \cite{KW} \cite{Lal} .

\section{Thermodynamic formalism over $\Pi(\cdot,\sigma)$ } \label{EP}

We assume that $c:X  \times \Omega\to \mathbb{R}$ is a Lipschitz function and we define the transfer operator $L_c$, which acts  on $C(\Omega)$, in the following way: given $\psi: \Omega\to \mathbb{R}$ we have
\[(L_c \psi )(y) = \sum_{x\in X}\sum_{\sigma(w)=y} e^{c(x,w)}\psi(w)=\sum_{x\in X}\sum_{a\in\{1,...,d\}} e^{c(x,ay)}\psi(ay).\]

\begin{proposition}\label{RPF} Let $c:X  \times \Omega\to \mathbb{R}$ be a Lipschitz function.
There exists a positive Lipschitz function $h$ and a positive number $\lambda$, such that,
\[ \sum_{x\in X}\sum_{\sigma(w)=y} e^{c(x,w)}h(w) =\lambda h(y).\]
Moreover, $\lambda$ is simple and the remainder of spectrum is contained in a disc with radius strictly smaller than $\lambda$.
\end{proposition}

\begin{proof} If $K>0$ is the Lipschitz constant of $c$, and we  consider the potential $ A(y) = \log\left(\sum_{x}e^{c(x,y)}\right)$, we have that $A$ is a Lipschitz function with constant $K$. \footnote{ $A(y)=\log\left(\sum_{x}e^{c(x,y)}\right)\leq\log\left(\sum_{x}e^{c(x,z)+K|y-z |}\right)=K|y-z |+A(z)$.} Then, $(L_{c}\psi)(y) =\sum_{\sigma(w)=y} e^{A(w)}\psi(w),$ and the proposition follows easily from classical arguments, see section 2 in \cite{PP}.
\end{proof}



We say that a Lipschitz cost (potential) $c$ is normalized if for any $y\in \Omega$, we have
\[\sum_{x\in X} \sum_{\sigma(w)=y}e^{c(x,w)} = 1.\]
If $c$ is Lipschitz and $\lambda$, $h$ are given by the Proposition \ref{RPF}, then $\overline{c}(x,y) = c(x,y) + \log(h(y))-\log(h(\sigma(y))) -\log(\lambda)$ is normalized.

Let us  assume now that $c:X \times \Omega \to \mathbb{R}$ is normalized. In this case we define the dual operator of the transfer operator $L_c$ in the following way: for a given probability $\nu$ on $\Omega$, we get a new probability $L_c^*(\nu)$ such that for any continuous function $\psi:\Omega\to\mathbb{R}$ we have

\begin{equation}\label{opdual}L_c^*(\nu)(\psi) = \int_{\Omega} L_c(\psi) \, d\nu= \int_{\Omega}\sum_{x\in X}\sum_{\sigma(w)=y} e^{c(x,w)}\psi(w) \, d\nu(y).\end{equation}

Using results of the classical thermodynamical formalism (see \cite{PP}) for the potential $A(y) = \log\left(\sum_{x}e^{c(x,y)}\right)$ we have
that the operator $L_c^*$ has a unique fixed point probability $\nu_c$, i.e. $L_c^*(\nu_c)=\nu_c$. We call $\nu_c$ the
{\it Gibbs probability measure} associated to the normalized cost $c$.


\bigskip
We want to extend the above definitions of transfer operator (and, moreover,  of the dual operator) which acts on functions of the variable $y$  to functions which depends on coordinates $(x,y)$. Let  $c$ be a normalized cost, then denote
\[\hat L_c(u)(y) = \sum_x\sum_{\sigma(w)=y} e^{c(x,w)}u(x,w)\] for any $u:X\times \Omega \to \mathbb{R}$.
Note that $\hat L_c$ sends $u:X \times  \Omega\to \mathbb{R}$ to the function denoted by $\hat L_c(u):\Omega \to \mathbb{R}$.
 For such normalized $c$ we denote  $\hat L_c^*$
 the operator on $P(X\times \Omega)$ defined by 
\begin{equation}\label{Op}
 \hat L_c^*(\pi)(u(z,y))\,=\,\int_{X\times \Omega} \left(\sum_x\sum_{\sigma(w)=y} e^{c(x,w)} u(x,w) \right) \, d\pi(z,y)\,.
\end{equation}


\begin{definition}

A probability on $X\times \Omega$ which is a fixed point for  $\hat L_c^*$ is called {\bf a Gibbs plan} for the normalized cost (potential) $c$.
It will be denoted by $\pi_c$.

\end{definition}

The normalization property implies that $\hat L_c^*(\pi)(1)=1$, i.e. the
$\hat L_c^*$ preserves the convex and compact set $P(X\times \Omega)$.

\begin{proposition} \label{Gi} Given a  normalized cost $c$ there exists a {unique} fixed point $\pi_c$ for the operator $\hat L_c^*$. We have that $\pi_c\in\Pi(\cdot,\sigma)$  and the $y$-marginal of $\pi_c$ is the  Gibbs measure $\nu_c$. Moreover, the support of $\pi_c$ is the set $X\times \Omega.$
\end{proposition}

We refer the reader to the appendix for the proof of the above result.

\bigskip

Now we define the entropy of plans (Definition \ref{entropia}), the pressure of a Lipschitz cost, and the concept of equilibrium plan for a cost (Definition \ref{def:pressure}). We will also see some examples and properties of such concepts, and prove a variational principle for the pressure which shows that the equilibrium plan for a cost is the Gibbs plan for the associated normalized cost (Theorem \ref{gibbs-equilibrium}). We will follow the main lines of  \cite{PP}, Chapter 3.

We denote by $[x,y_0...y_n]= \{(z,w)\in X\times\Omega:z=x,w_0=y_0,...,w_n=y_n\}$ and $[y_1...y_n]=\{w\in \Omega:w_0=y_1,...,w_{n-1}=y_n\}. $
Consider a fixed  plan  $\pi\in \Pi(\cdot,\sigma)$ with $y$-marginal $\nu$ and define $$J_\pi^{n}(x,y) = \frac{\pi([x,y_0...y_n])}{\nu([y_1...y_n])}$$ if $y=(y_1,y_2,...) \in \text{ supp}(\nu)$. From the Increasing Martingale Theorem the functions $J_\pi^{n}$ converge to a function $J_\pi(x,y)$ in $L^{1}(X\times\Omega,\mathcal{B}(X\times\Omega),\pi)$ and for $\pi$ a.e. $(x,y)$. For each plan $\pi$ this function $J_\pi$ can be also obtained via the Radon-Nikodyn Theorem. This is carefully explained in the appendix. We have, $J_\pi>0$ a.e. ($\pi$) and
$ \sum_x\sum_a J_\pi(x,ay) = 1$. For a plan  $\pi\in \Pi(\cdot,\sigma)$ we call $J_\pi$ the {\bf Jacobian of the plan}. For a general $\pi\in \Pi(\cdot,\sigma)$ the Jacobian $J_\pi$ is not necessarily Lipschitz but just measurable.


\begin{lemma}\label{L01} Suppose $\pi$ is a plan in $\Pi(\cdot,\sigma)$.
Then, for every $w\in C(X,\Omega)$,
\[\int_{X\times \Omega} \sum_{x}\sum_{a} J_\pi(x,ay)w(x,ay)\,d\pi = \int_{X\times \Omega} w(z,y)\, d\pi.\]
\end{lemma}

This result is proved 
 in the Appendix.

\medskip

\begin{lemma}\label{L2}
If a normalized Lipschitz potential $c(x,y)$ has a Gibbs plan $\pi_c$, then $J_{\pi_c} = e^c$.
\end{lemma}
\begin{proof}
This follows easily  from \cite{PP} proposition 3.2 and corollary 3.2.2 with simple adaptations in the computations.


\end{proof}

\noindent
\textbf{Example 1:}

We consider as an example the case where $X=\{1,2\}$, $\Omega=\{1,2\}^{\mathbb{N}}$, and $c$ is such that depends just on two coordinates on $y$, that is, $c(x,y) = c(x,\,y_1\,y_2)=c^{\,\,x}(y_1y_2)$, $x=1,2$, let us denote by $a^{i}_{r,s}= e^{ c^i(r,s)}$, $i,r,s=1,2$, and

\[A^{1}=\left(\begin{array}{cc}a^{1}_{11} & a^{1}_{12}\\ a^{1}_{21} & a^{1}_{22} \end{array}\right)=\left(\begin{array}{cc}1 & 1\\ 1 & 1 \end{array}\right)\,\,\,,\,\,
A^{2}=\left(\begin{array}{cc}a^{2}_{11} & a^{2}_{12}\\ a^{2}_{21} & a^{2}_{22} \end{array} \right)= \left(\begin{array}{cc} 1 & 1\\ 1 & 2 \end{array}\right).\]

We want to determine the Gibbs plan $\pi$ for such $c$.

The action of $L_c$ over potentials that depends only of one coordinate $y_1$, $y_1=1,2$, can be written in the form of the action on a vector $h$
\[(h_1,h_2)[A^{1}+A^{2}] = (h_1,h_2)\left(\begin{array}{cc}a^{1}_{11}+a^{2}_{11} & a^{1}_{12}+a^{2}_{12}\\ a^{1}_{21}+a^{2}_{21} & a^{1}_{22}+a^{2}_{22} \end{array}\right).\]
In this case $\lambda=\frac{5+\sqrt{17}}{2}$ and $(h_1,h_2) = (3+\sqrt{17},5+\sqrt{17})$ are respectively the associated eigenvalue and eigenvector. After normalization, we get the column stochastic matrix
\[\overline{A}=\left(\begin{array}{cc}\frac{a^{1}_{11}+a^{2}_{11}}{\lambda} & \frac{a^{1}_{12}+a^{2}_{12}}{\lambda}\frac{h_1}{h_2}\\ \\ \frac{a^{1}_{21}+a^{2}_{21}}{\lambda}\frac{h_2}{h_1} & \frac{a^{1}_{22}+a^{2}_{22}}{\lambda} \end{array}\right)
=\left(\begin{array}{cc} \frac{4}{5+\sqrt{17}} & \frac{4}{5+\sqrt{17}}\frac{3+\sqrt{17}}{5+\sqrt{17}}\\ \\ \frac{4}{3+\sqrt{17}} & \frac{6}{5+\sqrt{17}} \end{array}\right) =\left(\begin{array}{cc} 0,4384 & 0,3423\\ \\ 0,5616 & 0,6577 \end{array}\right) .\]

The stationary initial probability on $\{1,2\}$ of a stochastic matrix
\[\left(\begin{array}{cc}a & b\\ c & d \end{array}\right)\]
is given by $\nu=(p_1,p_2)^{T}=(\frac{b}{b+1-a}, \frac{1-a}{b+1-a})^{T}$. Finally,  we get $p_1=0,3786$ and $p_2=0,6213$.
In order to obtain the Gibbs plan we need to split $\overline{A}$ in the form  $\overline{A^{1}}$, $  \overline{A^{2}}$ where
\[\overline{A^{1}}= \left(\begin{array}{cc}\frac{2}{5+\sqrt{17}} & \frac{2}{5+\sqrt{17}}\frac{3+\sqrt{17}}{5+\sqrt{17}}\\ \\ \frac{2}{3+\sqrt{17}} & \frac{2}{5+\sqrt{17}} \end{array}\right) =\left(\begin{array}{cc} 0,2192 & 0,1711 \\ \\ 0,2808 &  0,2192 \end{array}\right) \,\,\,\]
\[\overline{A^{2}}= \left(\begin{array}{cc}\frac{2}{5+\sqrt{17}} & \frac{2}{5+\sqrt{17}}\frac{3+\sqrt{17}}{5+\sqrt{17}}\\ \\ \frac{2}{3+\sqrt{17}} & \frac{4}{5+\sqrt{17}} \end{array}\right) =\left(\begin{array}{cc} 0,2192 & 0,1711 \\ \\ 0,2808 &  0,4385 \end{array}\right)
.\]

This defines the Jacobian of the plan $\pi$ we are looking for.

Given, $u:X \times \Omega \to \mathbb{R}$, where $u(x,y)$ is such that depend just on the first coordinate of $y$ in the Bernoulli space, we have that
$\int u d\pi = u^{1}_1\pi^{1}_1+u^{1}_2\pi^{1}_2+u^{2}_1\pi^{2}_1+u^{2}_2\pi^{2}_2$. Using (\ref{Op}) we get
\[\pi^{1}=\overline{A^{1}}\nu= \left(\begin{array}{cc} 0,1893 \\ 0,2425\end{array}\right) \] and \[\pi^{2}=\overline{A^{2}}\nu = \left(\begin{array}{cc} 0,1893 \\ 0,3787\end{array}\right).\]

In this way we get the values $\pi^k_j=\pi(k, \overline{j})$, $j,k=1,2$, where $\overline{j}$ is the cylinder of size $1$ on $\{1,2\}^\mathbb{N}$ with first symbol $j$.

Note that as we get explicitly the Jacobian of $\pi$ one can obtain (via the use of (\ref{Op})) the probability of any cylinder.

In this way we get information on the  Gibbs plan $\pi$  we were looking for. It is easy to see that the above arguments can be applied in the same way for general matrices $A^1$ and $A^2$.

\bigskip

\begin{lemma} \label{L1a}
If $b$ is a normalized Lipchitz potential and $\pi\in \Pi(\cdot,\sigma)$, then
\[0\leq  - \int_{X\times \Omega} \log(J_\pi) \, d\pi \leq - \int_{X\times \Omega} b \, d\pi,\]
with equality,  if and only if, $b=\log(J_\pi)$.

Furthermore,
\[- \int_{X\times \Omega} \log(J_\pi) \, d\pi =\inf_{b\,normalized} - \int_{X\times \Omega} b \, d\pi.\]
\end{lemma}

The proof of this result appears in the Appendix.

\medskip

\begin{definition}\label{entropia}
If $\pi \in \Pi(\cdot,\sigma)$ we define the entropy of $\pi$ by
\[H(\pi) = -\int_{X\times \Omega} \log (J_\pi)\, d\pi =\inf_{b \, \text{normalized}} -\int_{X\times \Omega} b \, d\pi 
 .\]
\end{definition}
The functions $\log(J_\pi^{n})$ converge to $\log(J_\pi)$ in $L^{1}(X\times\Omega,\mathcal{B}(X\times\Omega),\pi)$ and we can compute the entropy from the limit
\[H(\pi) = -\lim_{n\to\infty}\int_{X\times \Omega} \log (J_\pi^{n})\, d\pi.\]
In the case $X$ has just one point the above definition matches the usual one for the Kolmogorov entropy (see \cite{Lop} and \cite{PP}).
If $X$ has $\#X$ elements and $\Omega=\{1,...,d\}^{\mathbb{N}}$, then $c(x,y) = -\log(d(\#X))$ is a  normalized cost, therefore
\[0 \leq H(\pi) \leq \log(d) + \log(\#X).\]

\begin{definition}\label{def:pressure}
The pressure of a Lipschitz continuous  cost (potential) $c$ is defined by
\[P(c) = \sup_{\pi \in \Pi(\cdot,\sigma)} \left(\int_{X\times\Omega} c\, d\pi + H(\pi)\right).\]

A plan $\pi\in \Pi(\cdot,\sigma) $ which realizes the supremum is called {\bf an equilibrium plan for $c$}.
\end{definition}

\begin{theorem}[Variational Principle over $\Pi(\cdot,\sigma)$]\label{gibbs-equilibrium}

 Let us fix a Lipschitz cost $c$. Then,
$P(c) =\log(\lambda_c)$, where $\lambda_c$ is the main eigenvalue of $L_c$. The equilibrium plan for $c$ is { unique} and given by the Gibbs plan for $\overline{c}:=c +\log(h_c) - \log(h_c\circ\sigma) - \log(\lambda_c)$, where $h_c$ is the eigenfunction associated to $\lambda_c$.

\end{theorem}

We refer the reader to the appendix for a proof.

\medskip

Now we present some properties of the entropy and pressure, as well as an example.

 As $X$ has a finite number of points we can consider the usual (non-dynamical) entropy of a probability measure $\mu \in P(X)$ given by $h(\mu) = -\sum_{x\in X}\mu(x)\log(\mu(x))$. For each $\sigma$-invariant probability $\nu \in P(\Omega)$ the Kolmogorov entropy is denoted by $h(\nu)$.

\begin{proposition}\label{entropia-prod} Given $\pi \in \Pi(\cdot,\sigma)$,  if the $x$-marginal of $\pi$ is a probability measure $\mu$ and the $y$-marginal of $\pi$ is an invariant measure $\nu$, then
\[H(\pi) \leq h(\mu) + h(\nu) .\]
If $\pi = \mu\times\nu$, then $H(\pi) = h(\mu)+h(\nu)$.
\end{proposition}

 For a proof see the Appendix.

 \medskip

\noindent
\textbf{Example 2:} If $X=\{1,2\}$ and $\Omega=\{1,2\}^{\mathbb{N}}$, then:\newline
1. Consider the plan $\pi$ defined from
\[\pi([x,y_0y_1...y_n]) = \left\{\begin{array}{cc}\frac{1}{2^{n+1}} & \text{if} \,y_0=x\\ 0 & \text{if}\,
y_0\neq x\end{array}\right.\]
The $x$-marginal of $\pi$ is $\mu=(\frac 1 2,\frac 1 2)$ and the $y$-marginal of $\pi$ is the Bernoulli measure $\nu$ with uniform distribution. Then, $h(\mu)=h(\nu)=\log(2)$ and $H(\pi)=\log(2)$. Indeed, we have
\begin{eqnarray*}
H(\pi)&=&\lim_{n\to\infty}-\int_{X\times \Omega} \log\left(\frac{\pi([x,y_0...y_n])}{\nu([y_1...y_n])}\right)d\pi(x,y)\\
&=&\lim_{n\to\infty}-\sum_x\sum_{y_0,...,y_n} \log\left(\frac{\pi([x,y_0...y_n])}{\nu([y_1...y_n])}\right)\pi([x,y_0...y_n])\\
&=& \lim_{n\to\infty}-\sum_{y_0...y_n} \log\left(\frac{\pi([y_0,y_0...y_n])}{\nu([y_1...y_n])}\right)\pi([y_0,y_0...y_n])\\
&=& \lim_{n\to\infty}-\sum_{y_0...y_n} \log\left(\frac{1}{2}\right)\frac{1}{2^{n+1}}=\log(2)
\end{eqnarray*}

2. Consider any plan $\pi =\mu\times \nu$ where $\mu=(\frac 1 2,\frac 1 2)$ and $\nu$ has support in a periodic orbit. In this case $h(\mu)=\log(2)$ while $h(\nu)=0$. Once more we get $H(\pi)=\log(2)$ because $\pi$ is a product plan.  \newline
\medskip

3. Let $\pi$ be the plan with two atoms $\pi = \frac{1}{2}\delta_{(1,\overline{12})}+\frac{1}{2}\delta_{(2,\overline{21})}$, where $\overline{y_0y_1}=(y_0,y_1,y_0,y_1,y_0,...)$. The $x$-marginal of $\pi$ is again $\mu=(\frac 1 2,\frac 1 2)$ while the $y$-marginal is the invariant measure with support equal to the periodic orbit $\overline{12}$. In the present case $H(\pi)=0$. Indeed,
\[\pi([x,y_0...y_n]) = \left\{\begin{array}{cc} \frac 1 2 & \text{if}\, x=y_0=y_2=y_4=... \,\text{and}\, x\neq y_1=y_3=...\\ 0 & \text{else}\end{array}\right. .\]
Then,
\begin{eqnarray*}
H(\pi) &=&  \lim_{n\to\infty}-\sum_x\sum_{y_0...y_n} \log\left(\frac{\pi([x,y_0...y_n])}{\nu([y_1...y_n])}\right)\pi([x,y_0...y_n])\\
&=&\lim_{n\to\infty}-\sum_{x} \log\left(\frac{\frac 1 2}{\frac 1 2}\right)\frac 1 2 =0.
\end{eqnarray*}

\begin{proposition} The entropy is a concave and upper semi-continuous function.
\end{proposition}
\begin{proof}
Both properties are  consequences of the definition $$H(\pi)=\inf_{b\,normalized} -\int_{X\times \Omega} b \,d\pi.$$ The proof follows the same arguments used in \cite{LMMS}.
\end{proof}


\begin{proposition}\label{pressure-properties}
The pressure has the following properties:

(a) 
if $c_1 \geq c_2$, then $P(c_1) \geq P(c_2)$.

(b)
$P(c+a)=P(c)+a$ if $a \in \mathbb{R}$.

(c) The pressure is convex .

(d) $|P(c_1)-P(c_2)| \leq \|c_1-c_2\|$.

\end{proposition}

The proof is in the Appendix.

\section{Kantorovich duality for Thermodynamic Formalism over $\Pi(\mu,\sigma) $ } \label{Ka}

In the last section $\mu$ was not fixed. In this section the probability $\mu$ on $X$ is fixed.

We define the $\mu$-pressure of $c$ by

\begin{equation}\label{mu-pressure}
 P_{\mu}(c)=\sup_{\pi \in \Pi(\mu,\sigma)}  \int_{X\times \Omega} c \, d\pi + H(\pi).
\end{equation}

Note that, $ P_{\mu}(c)\leq  P(c)$.



By compactness, there exists a  plan $\tilde\pi_c\in\Pi(\mu,\sigma)$ which attains the supremum
$$ \sup_{\pi \in \Pi(\mu,\sigma)}\int_{X\times \Omega}\, c(x,y) d \pi + H(\pi).$$

Remember that the Lipchitz functions are $C^0$ dense in the set of continuous functions.

The results that we will prove in this section can be resumed in the following.

\begin{theorem}[Variational Principle for $\Pi(\mu,\sigma)$]\label{varprinc}
\begin{equation}\inf_{\varphi:P(c-\varphi)= 0} \int_X \varphi(x) \, d\mu = \sup_{\pi \in \Pi(\mu,\sigma)} \int_{X\times \Omega}  c  \, d\pi\, +\, H(\pi). \label{varpri}
\end{equation}
The above infimum and supremum are attained in unique elements $\tilde \varphi$ and $\tilde \pi$. The maximizer $\tilde \pi$ is the Gibbs plan for $c-\tilde \varphi$.
\end{theorem}


The next corollary can be interpreted as the slackness condition on the present setting:

\begin{corollary}\label{corolmin}
Given $\varphi(x)$ such that $P(c-\varphi)=0$ and a plan $\pi_0 \in \Pi(\mu,\sigma)$, if $\pi_0$ is the Gibbs measure of $c-\varphi$, then $\pi_0$ attains the supremum and $\varphi$ attains the infimum in \eqref{varpri}.
\end{corollary}

\begin{proof}
We have
\begin{equation}\label{Pmu}
0=P(c-\varphi)=\sup_{\pi \in \Pi(\cdot,\sigma)} \int_{X\times \Omega}  (c - {\varphi}) \, d\pi\, +\, H(\pi)=\int_{X\times \Omega}  (c - {\varphi}) \, d\pi_0\, +\, H(\pi_0).
\end{equation}
Therefore
$$
\sup_{\pi \in \Pi(\mu,\sigma)} \int_{X\times \Omega}  c  \, d\pi\, +\, H(\pi) = \int \varphi d\mu\,
$$
and the supremum is attained in $\pi_0$.
\end{proof}

\begin{definition}
 Given a Lipschitz cost (potential) $c$ we define $\Phi_c$ as the set of all pairs of continuous functions $(\varphi,\psi)\in C(X)\times C(\Omega)$ which satisfy
\begin{equation}
\label{desigualdade phic}
\varphi(x)-\psi(y)+ (\psi \circ \sigma)(y)\geq  c(x,y)- b(x,y), \ \ \ \ \forall \, (x,y) \in X\times \Omega
\end{equation}
for some Lipschitz function $b$ with zero pressure. Following the classical terminology it is natural to call $\varphi(x)$ and $\psi(y)+ (\psi \circ \sigma)(y)$ of $c$-admissible pair.

\end{definition}

The theorem stated below is the version for positive temperature of the main theorem in \cite{LM} (which in some sense corresponds to zero temperature).



\begin{theorem}\label{dualidade} 

Given a Lipschitz cost $c$ we have
\begin{equation}\label{Fenchel-Rockafellar equation}
\,\,\,\,\
 \,\inf_{(\varphi,\psi)\in\Phi_c}  \int_X \varphi \, d\mu =  \sup_{\pi \in \Pi(\mu,\sigma)}\int_{X\times \Omega}\, c(x,y)\, d \pi + H(\pi).
\end{equation}
The supremum in ($\ref{Fenchel-Rockafellar equation}$) is attained in at least one plan.

\end{theorem}

In Proposition \ref{uniqueminimizer} we will prove that the infimum in
\eqref{Fenchel-Rockafellar equation} is attained in exactly one function $\varphi$ and that this infimum coincides with the left hand side of the Variation Principle stated above.  In order to prove this theorem we follow \cite{Vi1} and we use the next theorem (see also \cite{LM}).

\begin{theorem}[\textbf{Fenchel-Rockafellar duality}]\label{FR}
\label{Fenchel}
Suppose $E$ is  a normed vector space,  $\Theta$ and $\Xi$ two convex functions defined on $E$ taking values in $\mathbb{R}\cup \{+\infty\}$. Denote $\Theta^{\ast}$ and  $\Xi^{\ast}$, respectively, the Legendre-Fenchel transform of  $\Theta$ and $\Xi$.
Suppose there exists  $v_0\in E$, such that $\Theta(v_0)<+\infty,\, \Xi(v_0)<+\infty$ and that $\Theta$ is continuous on $v_0$.

Then,
\begin{equation}
\inf_{v \in E}[\Theta(v)+\Xi(v)]=\sup_{f\in E^{*}}[-\Theta^{*}(-f)-\Xi^{*}(f)] \label{rockafeller}
\end{equation}
Moreover, the supremum in ($\ref{rockafeller}$) is attained in at least one element in $E^*$.
\end{theorem}

\begin{proof} \textit{(of Theorem \ref{dualidade})}

It is enough to consider the case were $P(c)=0$. Indeed, let us we assume the theorem is proved for costs with zero pressure, if $P(c)\neq 0$, we define $\tilde c=c-P(c)$. In this way $(\varphi,\psi)\in\Phi_{ c}$, if and only if, $(\varphi-P(c),\psi)\in\Phi_{\tilde c}$. Then,
\[\inf_{(\varphi,\psi)\in\Phi_{ c}}  \int_X \varphi \, d\mu=\inf_{(\varphi-P(c),\psi)\in\Phi_{ \tilde c}}  \int_X \varphi \, d\mu=\inf_{(\tilde \varphi,\psi)\in\Phi_{ \tilde c}}  \int_X \tilde \varphi \, d\mu +P(c)\]
\[=\sup_{\pi \in \Pi(\mu,\sigma)}\int_{X\times \Omega}\,  \tilde c\, d \pi + H(\pi) +P(c) = \sup_{\pi \in \Pi(\mu,\sigma)}\int_{X\times \Omega}\,  c\, d \pi + H(\pi).\]
Hence, from now on we will assume that $P(c)=0$.

We want to use, the  Fenchel-Rockafellar duality in the proof. For this purpose we  define
$$E=C(X\times \Omega),$$
where $C(X\times \Omega)$ is the set of all continuous functions in $X\times \Omega$ taking values in $\mathbb{R}$, with the usual sup norm.
Moreover, $E^*=M (X \times \Omega)$ is the set of continuous linear operators in  $C(X\times \Omega)$ taking values in $\mathbb{R}$ with the  total variation norm. The elements in $M(X\times \Omega)$ are signed measures.

Define $\Theta,\Xi: E\longrightarrow \mathbb{R}\cup\{+\infty \}$ from
$$\Theta(u)=\left\{\begin{array}{ll}
0,&  \mbox{if} \ u(x,y)\geq c(x,y)-  b(x,y), \ \forall(x,y),\\ & \mbox{ for some }  b \mbox{ with } P(b)=0 ,\\
\\

+\infty,&  \mbox{in the other case}
\end{array}\right.
$$
%
%
%
and
$$
\Xi(u)=\left\{\begin{array}{ll}
\int_X\varphi \,d\mu,&  \mbox{if} \ u(x,y)=\varphi(x)-\psi(y)+(\psi \circ \sigma)(y), \\
&\mbox{where} \ (\varphi,\psi)\in C(X)\times C(\Omega)\,,\\
\\
+\infty,&  \mbox{in the other case}.
\end{array}\right.
$$
{Note that  $\Xi$ is well defined. Indeed, if  $u(x,y)= \varphi_{1}(x)-\psi_{1}(y)+(\psi_{1} \circ \sigma)(y) =\varphi_{2}(x)-\psi_{2}(y)+(\psi_{2} \circ \sigma)(y) $,
then, integrating over a probability in $\Pi(\mu,\sigma)$, we conclude that
$\int_X\varphi_{1}(x)d\mu=\int_X\varphi_{2}(x)d\mu
,\, \forall \, x\in X$.}
\bigskip

 Now we will show that the hypothesis in Theorem \ref{FR} are satisfied. 
The convexity of $\Xi$ is immediate. To show  the convexity of $\Theta$
  take $u_1$ and $u_2$ such that $\Theta(u_1)=\Theta(u_2)=0$, then there exist $b_1$ and  $b_2$ Lipschitz with $P(b_1)=P(b_2)=0$, such that, $u_1\geq c-b_1$ and $u_2\geq c-b_2$. Note that
$$\lambda u_1+(1-\lambda)u_2\geq c-(\lambda b_1+(1-\lambda)b_2),$$
and then using item (c) of  Proposition \ref{pressure-properties}, we see
 $$a=P(\lambda b_1+(1-\lambda)b_2) \leq \lambda P(b_1)+(1-\lambda)P(b_2)=0.$$
Therefore, by item (b) of Proposition \ref{pressure-properties},   $P(\,[\lambda b_1+(1-\lambda)b_2\,]-\,a )=0.$
In this way
$$\lambda u_1+(1-\lambda)u_2\geq c-(\lambda b_1+(1-\lambda)b_2)\geq c-([\lambda b_1+(1-\lambda)b_2]-a).$$
This shows that $\Theta(\lambda u_1+(1-\lambda)u_2)=0$ and hence $\Theta$ is convex.
Finally we exhibit $u_0$ in the domain of  $\Xi$ and $\Theta$: take $u_0=1$. Then, $\Xi (1)=1$ and $\Theta(1)=0$, because  $ 1>0=c-c$, and $P(c)=0$. If $w\in C(X\times\Omega)$ such that $\|w-1\|<1/2$, then $w>0=c-c$, and $P(c)=0$. This shows that $\Theta$ is continuous in $u_0.$

\medskip

Let us now compute the Legendre-Fenchel transform of $\Theta$ and $\Xi$.
We denote by  $\pi$ an element in $ E^{*}=M(X \times\Omega)$.

For any  $\pi\in E^{*}$, by the definition of $\Theta$
we get
\begin{eqnarray*}
&&\Theta^*(-\pi) =\sup_{u\in E}\left\{\langle -\pi,u \rangle- \Theta(u) \right\}\\
&&= \sup_{u\in E}\left\{\langle\pi,u\rangle : \ u\leq - c + b,\,\, b \mbox{ with } P(b)=0 \right\}.
\end{eqnarray*}
If $\pi$ is not a positive functional, then there exists a function $v \leq 0$, $v\in C(X\times \Omega)$, such that, $\langle \pi, v\rangle > 0$. We can assume that $v$ is Lipschitz. Note that $c+v\leq c$, hence using item (a) of Proposition \ref{pressure-properties} and that $P(c)=0$, we have $P(c+v)\leq 0$. Therefore, $v=-c+(c+v)\leq -c+(c+v-P(c+v))$, with $P(c+v-P(c+v))=0$. Taking  $u=\lambda v$, and considering  $\lambda\rightarrow +\infty$, we get that
$$\sup_{u\in C(X\times \Omega)}\left\{\langle \pi, u\rangle: \, u\leq - c + b , \,\, \, \, b \mbox{ with } P(b)=0\right\}=+\infty.$$
Therefore, we assume from now on that $\pi$ is a positive functional.
Note that, if $\pi\in M^+(X\times \Omega)$, then
$$\sup_{u\in C(X\times \Omega)}\left\{\langle \pi, u\rangle: \, u\leq - c + b , \,\, \, \, b \mbox{ with } P(b)=0\right\}=\langle \pi, -c\rangle+\sup_{b\,: \,  P(b)=0}\langle \pi, b\rangle.  $$
Hence, we obtain
\begin{eqnarray}
\label{tetaestrela}
\Theta^*(-\pi)=\left\{\begin{array}{ll}
\displaystyle\langle \pi,\,( - c \,)\,\rangle +\sup_{b\,: \,  P(b)=0}\langle \pi, b\rangle\,,\ \mbox{if} \ \pi\in M^+(X\times \Omega) \\
\\
+\infty, \ \mbox{in the other case}.
\end{array}\right.
\end{eqnarray}

Analogously, by the definition of $\Xi$ we get that
$$\begin{array}{l}
\Xi^*(\pi)=\displaystyle{\sup_{u\in E}}\left\{\langle \pi,u \rangle - \Xi(u)\right\} \\ \\
=\displaystyle{\sup_{u\in E}}\left\{\begin{array}{lll} \,
\langle \pi, u \rangle  -\int_X \varphi\, d\mu   : \\
u(x,y)=\varphi(x)-\psi(y)+ \psi (\sigma (y)) \ \mbox{where} \ (\varphi,\psi)\in C(X)\times C(\Omega)
\end{array}\right\} \\ \\
=\displaystyle{\sup_{ (\varphi,\psi)\in C(X)\times C(\Omega),}\left\{ \langle \pi, \varphi(x)- \psi(y)+ \psi (\sigma (y)) \rangle- \int_X \varphi\, d\mu   \right\}.}
\end{array}$$

Note that, if $\langle \pi, \varphi(x) \rangle > \int_X \varphi \,d\mu$,    for some $\varphi$, choosing
$\lambda.\varphi$ with $\lambda \to\infty$, the supremum will be equal to $+\infty$. Also if
$\langle\pi, \psi(y)- \psi (\sigma (y))\rangle > 0,$ for some $\psi$, taking $\lambda.\psi$ with $\lambda \to \infty$, the supremum will be $+\infty.$ The case where we consider  the other inequality is analogous.
Then, we can assume that $\langle \pi, \varphi(x) \rangle = \int_X \varphi \, d\mu$ and $\langle \pi, \psi(y)- \psi(\sigma(y))\rangle= 0$.

In order to simplify the notation, we define
\begin{eqnarray*}
  \Pi^*(\mu) = \left\{  \pi \in M(X\times \Omega): \, \begin{array}{ll} \langle \pi, \varphi(x)\rangle=\int_X \varphi \,d\mu \, \, \text{and} \, \langle \pi, \psi(y)- \psi(\sigma(y))\rangle= 0 \\
     \, \forall (\varphi,\psi) \in C(X)\times C(\Omega)\end{array} \right\}.
\end{eqnarray*}

{With this notation we can write}
\begin{equation}
\label{xiestrela}
\Xi^*(\pi)=\left\{\begin{array}{lll}
\ \displaystyle{0, \ \mbox{if} \  \pi\in\Pi^*(\mu),}\\
\\
+\infty, \ \mbox{in the other case}.
\end{array}\right.
\end{equation}

We observe that
 if  $\pi \in  M^+(X\times \Omega)\cap\Pi^*(\mu)$, then
$\langle \pi, 1\rangle = \mu(1) = 1$, $\langle \pi, u \rangle \geq 0$ when
 $u \geq 0$ and also
$\langle \pi, \cdot\rangle$ is linear.
{From these properties we get that $\pi \in P(X\times \Omega)$. Moreover, by definition of $\Pi^*(\mu)$, the $x$-marginal of $\pi$  is $\mu$ and the $y$-marginal of $\pi$ is $\sigma$-invariant. Hence, we conclude $M^+(X\times \Omega)\cap\Pi^*(\mu) = \Pi(\mu,\sigma)$.}

By definition \ref{entropia}, if $\pi\in\Pi(\mu,\sigma)$, we get  $$\displaystyle -\sup_{b\,: \,  P(b)=0} \pi( b)= -\sup_{b \text{ normalized }} \int_{X\times \Omega}  b \, d\pi=H(\pi).$$

The left hand side of (\ref{rockafeller}) is given by
\begin{eqnarray*}
& &\inf_{u\in E}[\Theta(u)+\Xi(u)]\\
& & =\inf_{u\in E}\bigg\{\int_X\varphi \,d\mu :  \varphi(x)-[\psi -(\psi \circ \sigma)](y)\geq c(x,y) - b(x,y),\\
&&\text{ for some } b \text{ with } P(b)=0, (\varphi,\psi)\in C(X)\times C(\Omega)\bigg\}\\
&&=\inf_{(\varphi,\psi)\in\Phi_c } \,\, \int_X \varphi \, d\mu .
\end{eqnarray*}

The right hand side of (\ref{rockafeller}) is given by
\begin{eqnarray*}
&&\sup_{\pi\in E^{*}}[-\Theta^{*}(-\pi)-\Xi^{*}(\pi)]\\
&&=\sup_{\pi \in E^{*}}\left\{
\begin{array}{ll}
\displaystyle \pi(c) +H(\pi)\, \,   ,
 & \mbox{if} \ \pi\in\Pi(\mu,\sigma) \\
\\
-\infty,  &  \mbox{in the other case}
\end{array}
\right\}\\
&&= \sup_{\pi\in \Pi(\mu,\sigma)} \{\pi(c)+H(\pi) \}.
\end{eqnarray*}
 \newline

Therefore,  from (\ref{rockafeller}) we get
\begin{equation*}
\inf_{(\varphi,\psi)\in\Phi_c} \int_X \varphi \, d\mu = \sup_{\pi \in \Pi(\mu,\sigma)} \int_{X\times \Omega}  c  \, d\pi\, +\, H(\pi).
\end{equation*}

Theorem \ref{Fenchel} claims that%
 $$\sup_{f \in E^*}[-\Theta^*(-f)-\Xi^*(f)]=\sup_{\pi \in \Pi(\mu,\sigma)} \int_{X\times \Omega} c \, d\pi+H(\pi).$$%
is attained, for at least one element (but we already know this by compactness).

\end{proof}

\bigskip

\noindent
\textbf{Example 3:}

We consider again as an example the case where $X=\{1,2\}$, $\Omega=\{1,2\}^{\mathbb{N}}$, and $c$ is such that depends just on two coordinates on $y$, that is, $c(x,y) = c(x,\,y_1\,y_2)=c^{\,\,x}(y_1y_2)$, $x=1,2$, and

\[A^{1}=\left(\begin{array}{cc}a^{1}_{11} & a^{1}_{12}\\ a^{1}_{21} & a^{1}_{22} \end{array}\right)\,\,\,,\,\,
A^{2}=\left(\begin{array}{cc}a^{2}_{11} & a^{2}_{12}\\ a^{2}_{21} & a^{2}_{22} \end{array} \right),\]
where $a^{i}_{r,s}= e^{ c^i(r,s)}$, $i,r,s=1,2$.

We fix $\mu=(\mu_1,\mu_2)$ and we are going to explain how one can get the solution $\pi\in \Pi(\mu,\sigma)$ of the above transport problem via the equation
\[\inf_{\varphi:P(c-\varphi)= 0} \int_{X} \varphi(x) \, d\mu = \sup_{\pi \in \Pi(\mu,\sigma)} \int_{X\times \Omega}  c  \, d\pi\, +\, H(\pi).\]

We consider first the left side expression.

The function $\varphi$ is described by $(\varphi_1,\varphi_2).$ The condition $P(c-\varphi)=0$ means that

\[ \left(\begin{array}{cc}         e^{c^{1}_{11}} \,e^{-\varphi_1} +e^{c^{2}_{11}}\,e^{-\varphi_2} & e^{c^{1}_{12}} \,e^{-\varphi_1} +e^{c^{2}_{12}}\,e^{-\varphi_2}\\ e^{c^{1}_{21}}\,e^{-\varphi_1}+ e^{c^{2}_{21}}\,e^{-\varphi_2} & e^{c^{1}_{22}}\,e^{-\varphi_1}+ e^{c^{2}_{22}}\,e^{-\varphi_2} \end{array}\right)\]
has a dominant  eigenvalue $1$.

We are interested in the $z_1= e^{-\varphi_1}$, $z_2= e^{-\varphi_2}$ which are solutions of the equation
$$ \det \, \left(\begin{array}{cc}         (e^{c^{1}_{11}} \,z_1   +e^{c^{2}_{11}}\,z_2)\,-1 & e^{c^{1}_{12}} \,z_1 +e^{c^{2}_{12}}\,z_2\\ e^{c^{1}_{21}}\,z_1+ e^{c^{2}_{21}}\,z_2 & (e^{c^{1}_{22}}\,z_1+ e^{c^{2}_{22}}\,z_2)-1 \end{array}\right)\,=0.$$

In this way we get that $(z_1,z_2)$ describes an algebraic curve on $\mathbb{R}^2$. This equation does not discriminate if
the eigenvalue $1$ is maximal but this is not a big problem. Now we have to find the points $(z_1,z_2)$ of such curve such that its normal vector is
colinear with the vector $v(z_1,z_2)=(\mu_1\frac{1}{z_1},\mu_2\frac{1}{z_2})$, which is the gradient of the function $(z_1,z_2) \to \log z_1 \mu_1 + \log z_2 \mu_2$ (Lagrange multipliers). This will determine a finite set (in the generic case) of possible $\varphi=(\varphi_1,\varphi_2)$, which are critical points for $\varphi \to \int \varphi\, d\mu$. We test these possibilities and then we get the minimal $\varphi$ which we denote by $\tilde{\varphi}$. In this way we determine the left hand side of the last main equality and the value of the $\mu$ pressure of $c$.

Now we consider the potential $\tilde{c}= c-\tilde{\varphi}$. Finally using the same procedure of example 1 one can get the Gibbs plan for $\tilde{c}$. In this way we solve the Ergodic Transport problem for $c$  with a fixed marginal $\mu$.

\bigskip

To show the uniqueness in  the next proposition we will use the property that the pressure is an analytical function of the potential (see \cite{Ru} and \cite{SSS}).

\begin{proposition}\label{uniqueminimizer}
\begin{equation}
\inf_{(\varphi,\psi)\in\Phi_c} \int_X \varphi \, d\mu=\inf_{\varphi:P(c-\varphi)= 0} \int_X \varphi \, d\mu \end{equation}
%
The infimum is attained  at exactly  one function $\tilde{\varphi}$.
\end{proposition}
\begin{proof} 

 Given $(\varphi,\psi)\in\Phi_c$, there exists $b$ such that $P(b)=0$ and
$$c(x,y)- \varphi(x)+\psi(y)- (\psi \circ \sigma)(y) \leq b(x,y).$$
 Then, by item (a) of proposition \ref{pressure-properties},  $P(c(x,y)- \varphi(x))\leq 0$. On the other hand, if $P(c(x,y)- \varphi(x))=a\leq 0$, then we  define $b(x,y)=c(x,y)- \varphi(x)-a$. We have that $P(b)=0$ and $b(x,y)\geq c(x,y)- \varphi(x)$. Hence,
\begin{equation*}\label{infimo}
\inf_{(\varphi,\psi)\in\Phi_c} \int_X \varphi \, d\mu=\inf_{\varphi:P(c-\varphi)\leq 0} \int_X \varphi \, d\mu .
\end{equation*}

By monotonicity of the pressure, we have
$$\inf_{(\varphi,\psi)\in\Phi_c} \int_X \varphi \, d\mu=\inf_{\varphi:P(c-\varphi)= 0} \int_X \varphi \, d\mu .  $$
%

%
Note also that, if $P(c-\varphi)\neq 0$, we can add the constant $-P(c-\varphi)$ and get $P(c-\varphi - P(c-\varphi))=0$. Then,
\begin{equation}\label{infpres}\inf_{\varphi:P(c-\varphi)= 0} \int_X \varphi \, d\mu  = \inf_{\varphi} \int_X \varphi\, d\mu + P(c-\varphi)=  \inf_{\varphi} -\int_X \varphi\, d\mu + P(c+\varphi)
\end{equation}
Consider the continuous function $F:C(X) \to \mathbb{R}$ given by
\[F(\varphi)= -\int_X \varphi\, d\mu + P(c+\varphi),
\]
we see that, if $a \in \mathbb{R}$, then $F(\varphi+a)=F(\varphi)$. This shows we can minimize $F(\varphi)$ among $\varphi$ such that $\varphi(0)=0$.

In order to prove the uniqueness of the minimizer of $F$, we assume that $X=\{0,...,k\}$ has $k+1$ elements,   $\varphi(0)=0$ and we identify $\varphi$ with $v\in\mathbb{R}^{k}$, in the following way, $\varphi=(0,v_1,...,v_k)$, i.e., $\varphi(j) = v_j, \,j=1,2,...,k$.

Therefore, $F:\mathbb{R}^k \to \mathbb{R}$
associates to each vector 
$v \in \mathbb{R}^k$ the number
\[F(v)= -\int_X \varphi\, d\mu + P(c+\varphi).\]

In this way, to finish the proof, we need  to  show that $F(v)$ has only one minimizer $\tilde{v}= \tilde{v}_c$.


We begin by proving that when $t\to+\infty$, $F(tv)\to+\infty$ 
uniformily in $\mathbb{S}^{k-1}$, i.e, there exist an $\epsilon>0$ and $\xi \in \mathbb{R}$, such that, for any $v \in \mathbb{S}^{k-1}$, we have
that
\begin{equation}\label{explodeuniformemente}
F(tv)\geq t\varepsilon + \xi.
\end{equation}

In order to do that, let $K_1=\{v\in\mathbb{S}^{k-1} : v_i\geq 0 \text{ for some } i   \}$ and $K_2=\{v\in\mathbb{S}^{k-1} : v_i\leq 0 \,\,\,\forall \,i   \}$. We have $K_1\cup K_2=\mathbb{S}^{k-1} $ and $K_1,K_2$ are compact sets.

Using the fact that  $\text{supp}(\mu)=X$, we have that the functions $-\int_X \varphi\, d\mu +\max_i \varphi(i)$ and $-\int_X \varphi \,d\mu $ are continuous and strictly positive in $K_1$ and $K_2$, respectively, where $\varphi=(0,v_1,...,v_k)$. Then,  there exists $\varepsilon>0$, such that, $-\int_X \varphi\, d\mu +\max_i \varphi(i)\geq \varepsilon$, for all $v\in K_1$, and   $-\int_X \varphi \,d\mu \geq \varepsilon$, for all $v\in K_2$.

Let us take $v\in K_1$ and  a plan $\pi$ with $x$-marginal $\delta_k$, such that  $\varphi(k)=\max_i \varphi(i)$, therefore $P(t\varphi)\geq \int_{X\times \Omega} t\varphi \,d\pi+H(\pi)= t\varphi(k)+H(\pi)\geq t\varphi(k)=
t \max_i \varphi(i) $.
Hence,
\begin{eqnarray*}
F(tv)&=&-\int_X t\varphi \,d\mu +P(c+t\varphi)\geq -\int_X t\varphi\, d \mu +P(t\varphi) +\min c\\
&\geq& -\int_X t\varphi \,d \mu + t \max_i \varphi(i) +\min c\geq t\varepsilon +\min c.
\end{eqnarray*}
Now, take $v\in K_2$ and a  plan  $\pi$ with $x$-marginal $\delta_0$, we have
$P(t\varphi)\geq \int_{X\times \Omega} t\varphi \,d\pi+H(\pi)= H(\pi)\geq 0 $.
Hence,
\begin{eqnarray*}
F(tv)&=&-\int_X t\varphi \,d\mu +P(c+t\varphi)\geq -\int_X t\varphi\, d \mu +P(t\varphi) +\min c\\
&\geq& -\int_X t\varphi \,d \mu   +\min c\geq t\varepsilon +\min c.
\end{eqnarray*}

We conclude that \eqref{explodeuniformemente} holds,
and this shows that $F$ assume a minimum
$\tilde{v}$ in $\mathbb{R}^k$.

Now we will prove that the minimizer $\tilde{v}$ is unique.  Note that $F$ is well defined for any $v$. We want to show that $F$ is locally analytic.
It will be the restriction of a complex analytic function.
 We  use the analyticity of the pressure, which imply that  $F$ is analytic on $v$.
Indeed, note that $A_v(y)=A_\varphi(y) = \log\left(\sum_{x}e^{c(x,y)+ \varphi(x)}\right)$ is an analytic function on $v$ (locally can be extended to a complex analytic function) taking values on the Banach space of Holder potentials on the variable $y$. As the composition of analytic functions is  also analytic and the pressure is analytic on the potential (see Theorem 5.26 in \cite{Ru}) we get our claim. As $F$ is globally defined and locally analytic then it is analytic in the all domain.

We also know that the pressure is convex as a function of $c$ (see Proposition \ref{pressure-properties}). This implies that $F$ is also convex in $v$.

Suppose that $\tilde{v}$ and $\hat{v}$ are minimizers for $F$. Using the convexity of $F$,  we know that all convex combinations of $\tilde{v}$ and $\hat{v}$ are minimizers for $F$.

Now let the function $G:\mathbb{R}\to\mathbb{R}$ be defined by $G(t)=F(\tilde{v} + t (\hat{v}-\tilde{v}))$.
$G$ is an analytical function which converges to $+\infty$ when $t \to \pm \infty$.

Note that the second derivative of $F$ can not be 0.
Therefore, it can not be constant in a open interval of the real line, and we conclude that $\tilde{v}=\hat{v}$.

\end{proof}

\begin{corollary}\label{corolmin}
Let $\tilde{\varphi}$ be the unique minimizer for the
Fenchel-Rockafellar duality
\eqref{Fenchel-Rockafellar equation}, then
the Gibbs plan $\pi_{c- \tilde{\varphi}}$,
for $c- \tilde{\varphi}$, belongs to $\Pi(\mu,\sigma)$ and is the unique maximizer of \eqref{Fenchel-Rockafellar equation}.
\end{corollary}

\begin{proof}
  Let $\tilde \varphi$ be the minimizer of \eqref{Fenchel-Rockafellar equation} then
  \begin{equation}\label{mindual}
\int_X \tilde{\varphi} \,d\mu =
\sup_{\pi \in \Pi(\mu,\sigma)} \int_{X\times \Omega}  c  \, d\pi\, +\, H(\pi),
\end{equation}
hence
\begin{equation}\label{Pmu}
\sup_{\pi \in \Pi(\mu,\sigma)} \int_{X\times \Omega}  (c - \tilde{\varphi}) \, d\pi\, +\, H(\pi)=0,
\end{equation}
which implies that $P_{\mu}(c-\tilde{\varphi})=0$.
 Also, by Proposition \ref{uniqueminimizer}
we know that $P(c-\tilde{\varphi})=0$.
 Therefore, $P_{\mu}(c-\tilde{\varphi})= P(c-\tilde{\varphi})=0$.

Now, let $\pi_{\mu}$ be a plan   that attains the supremum in \eqref{Pmu}, which exists by Theorem \ref{dualidade}, then $\pi_{\mu}$ also attains the supremum in $P(c-\tilde{\varphi})=0$. Finally  using Theorem \ref{gibbs-equilibrium}, we see that $\pi_{\mu}=\pi_{c- \tilde{\varphi}}$, as $\pi_{c- \tilde{\varphi}}$ is the unique equilibrium plan for $c-\tilde{\varphi}$, and this implies $\pi_{c- \tilde{\varphi}}\in\Pi(\mu,\sigma)$ and that $\pi_{c- \tilde{\varphi}}$ is the unique maximizer of \eqref{Pmu}, and hence of \eqref{mindual}.
\end{proof}

\begin{proof}\textit{(of Theorem \ref{varprinc})}
It follows by Theorem \ref{dualidade} and Proposition \ref{uniqueminimizer} that there exists a unique $\tilde \varphi$ such that
\begin{equation}\label{mindual2}
\int_X \tilde{\varphi} \,d\mu =\inf_{\varphi:P(c-\varphi)= 0} \int_X \varphi \, d\mu=
\sup_{\pi \in \Pi(\mu,\sigma)} \int_{X\times \Omega}  c  \, d\pi\, +\, H(\pi),
\end{equation}
and by Corollary \ref{corolmin}, we see that $\pi_{c- \tilde{\varphi}}$ is the unique maximizer of \eqref{mindual2}.
\end{proof}


\bigskip

\bigskip


 \bigskip

\bigskip

\section{The zero temperature limit} \label{zero}

In this section we show that the main result proved in \cite{LM} and discussed in the introduction can be obtained from the reasoning of the above section considering the zero temperature limit.

\begin{center}
{\bf Zero temperature for $\Pi(\cdot,\sigma)$}
\end{center}

Given a  Lipschitz potential $c$ and a real variable $\beta>0$, consider the potential $\beta \, c$. The parameter $\beta$ corresponds to the inverse of the temperature in the Thermodynamic Formalism.
We denote by $\lambda_{\beta}$ the main eigenvalue of $L_{\beta c}$ and by $h_{\beta }$ the main eigenfunction associate to $\lambda_{\beta}$ (we can suppose that $\min(h_\beta)=1$ for any $\beta$). Denote also by $\pi_\beta$ the equilibrium plan for $\beta c$.

We note that $h_\beta$ is the positive eigenfunction of the  Ruelle operator with potential
$$A_\beta(y)=\log\left(\sum_x e^{\beta c(x,y)}\right),$$
in the classical Thermodynamic Formalism
setting. Hence the Lipschitz constant of $\log(h_\beta)$ increase linearly with $\beta$ \cite{BCLMS}. From the Arzela-Ascolli Theorem $\frac{1}{\beta_n}\log(h_{\beta_n})$ converges for some sequence $\beta_n\to\infty$.
Let $$m = \sup_{\pi\in \Pi(\cdot,\sigma)}\int_{X\times \Omega} c(x,y)  \, d\pi.$$
Using Theorem \ref{gibbs-equilibrium} and that $0\leq H(\pi)\leq \log(\#X)+\log(d)$,
we have
\[ \beta m \leq \log(\lambda_\beta) \leq \beta m + \log(\#X)+\log(d)\]
and then $\lim_{\beta\to\infty} \frac{\log(\lambda_\beta)}{\beta}=m$.
By compactness we know that there exist convergent sub-sequences of $\pi_\beta$, $\beta\to\infty$.

Suppose that for some sequence $\beta_n$ we have $\frac{1}{\beta_n}\log(h_{\beta_n}) \to V$ and $\pi_{\beta_n}\to\pi_\infty$.
Applying the Laplace's Method (see \cite{BCLMS} and \cite{LMMS}) on the equation
\[\sum_x\sum_a e^{\beta c(x,ay) + \log(h_{\beta}(ay)) - \log(h_\beta(y)) - \log(\lambda_\beta)} = 1,\]
we conclude that
\[\sup_{x}\sup_{a} [\,c(x,ay) + V(ay) - V(y) - m] = 0, \,\, \,\, \,\, \forall\, y.\]
Let us prove
that $\pi_{\infty}$ is a maximizing measure for $c$:
analyzing  the equation
\[ \int_{X\times \Omega} \beta c(x,y) \, d\pi_\beta + H(\pi_\beta)= \sup_{\pi\in \Pi(\cdot,\sigma)}  \int_{X\times \Omega} \beta c(x,y)  \, d\pi + H(\pi),\]
 we conclude that (dividing by $\beta_n$, making $\beta_n\to\infty$ and using that $H$ is a bounded function)
\[ \int_{X\times \Omega} c(x,y) \, d\pi_\infty =\sup_{\pi\in \Pi(\cdot,\sigma)}\int_{X\times \Omega} c(x,y)  \, d\pi=m.\]
Now we prove the duality between the primal and dual problem:
\begin{theorem} $m$ is the smallest real number $\alpha$, such that, there exists a continuous function $S:\Omega \to\mathbb{R}$, satisfying
$c(x,y) +S(y) -S(\sigma(y)) -\alpha \leq 0$, for any $x\in X$ and $y\in\Omega$.

In this way, if we define $\Phi_c$ as the set of pair $(\alpha, S)$ such that $-\alpha +S(y)-S(\sigma(y)) \leq -c(x,y)$ then
\[\inf_{(\alpha,S)\in \Phi_c} \alpha = \sup_{\pi\in \Pi(\cdot,\sigma)}\int_{X\times \Omega} c(x,y)  \, d\pi\]
or, equivalently
\[\sup_{(\alpha,S)\in \Phi_c} -\alpha = \inf_{\pi\in \Pi(\cdot,\sigma)}\int_{X\times \Omega} -c(x,y)  \, d\pi.\]
\end{theorem}

\begin{proof}
From the above arguments we know that
\[\sup_{x}\sup_{a} [\,c(x,ay) + V(ay) - V(y) - m ]= 0, \,\, \,\, \,\, \forall\, y\]
proving that $m$ is a possible number. In order to show that $m$ is the smallest possible  number, fix $\alpha$ and $S$ such that $c(x,y) +S(y) -S(\sigma(y)) -\alpha \leq 0$ for any $x\in X,y\in\Omega$. Then
\[\int_{X\times \Omega} c\, d\pi \leq \alpha, \,\,\,\,\, \forall \, \pi \in \Pi(\cdot,\sigma).\]
In this way
\[m=\sup_{\pi\in\Pi(\cdot,\sigma)} \int_{X\times \Omega} c\, d\pi \leq \alpha.\]
\end{proof}

\begin{center}
{\bf Zero temperature for $\Pi(\mu,\sigma)$}
\end{center}

Now we consider the analogous  problem over $\Pi(\mu,\sigma)$. For each $\beta>0$, given the potential $\beta c$, by Theorem \ref{varprinc}, there exists a unique function $\varphi_\beta(x)$ such that $P(\beta c-\varphi_\beta)=0$ and
\[ \int_X \varphi_\beta\, d\mu = \sup_{\pi\in\Pi(\mu,\sigma)} \int_{X\times \Omega} \beta c\, d\pi + H(\pi).\]
Let $h_\beta$ be the eigenfunction associate to the eigenvalue $1$ for $L_{\beta c- \varphi_\beta}$.%
We suppose $\min(h_\beta)=1$.
Let $\pi_\beta \in \Pi(\mu,\sigma)$ be the equilibrium plan for $\beta c - \varphi_\beta$.

Now we want to prove that the sequences $\frac{\varphi_\beta}{\beta}$ and $\frac{h_\beta}{\beta}$ converge in subsequence, when $\beta\to\infty$.
 To do this we need show that the Lipschitz constant of $\beta c-\varphi_\beta$ increases linearly with $\beta$.

 According to \eqref{infpres}  we can add a constant to $\varphi_\beta$ and take $\varphi_\beta^{*}$ a minimizer of
 \[ \int_X \varphi\, d\mu + P(\beta c-\varphi),\]
 such that $\varphi_\beta^{*}(0)=0$, we have that $ \varphi_\beta =  \varphi_\beta^{*}+P(\beta c - \varphi_\beta^{*})$.

As in the proof  of Proposition \ref{uniqueminimizer} we suppose $X=\{0,...,k\}$, we consider for each 
$\beta c $ the  function $F(\varphi)=-\int_X \varphi\,d\mu +P(\beta c+\varphi)$ 
 and we see, by the same arguments, that $-\varphi_{\beta}^{*}$ is a minimizer of $F$, in particular $F(-\varphi_{\beta}^{*})\leq F(0)=P(\beta c)$.
Also, there exists $\epsilon>0$ such that $F(t\varphi)>t\epsilon+\min(\beta c)$ for any $\varphi\in \mathbb{S}^{k-1}$. Therefore, if $t> \frac{P(\beta c)-\min(\beta c)}{\epsilon}$ then $F(t\varphi)>F(0)\geq F(-\varphi_{\beta}^{*}),$ for any $\varphi\in \mathbb{S}^{k-1}$. If we write $-\varphi_\beta^{*}=\| \varphi_\beta^{*}\| \tilde\varphi_\beta^{*} ,$ with $\tilde\varphi_\beta^{*}\in \mathbb{S}^{k-1} $, we see that
$\|\varphi_\beta^{*}\| \leq \frac{P(\beta c)-\min(\beta c)}{\epsilon} $ and then  $\frac{\|\varphi_\beta^{*}\|}{\beta} \leq \frac{P(\beta.c)}{\beta \epsilon}-\frac{\min(c)}{\epsilon} $. From the arguments above we know that $\frac{P(\beta c)}\beta$ converges to $\sup_{\pi\in\Pi(\cdot,\sigma)}\int_{X\times \Omega} c\, d\pi$. Then, there exists a constant $K$ such that $\frac{\|\varphi_\beta^{*}\|}{\beta} \leq K$.

Now we claim that $\frac{\varphi_\beta}{\beta}$ is bounded. Indeed, as  $\frac{\varphi_\beta}{\beta} = \frac{\varphi_\beta^{*}}{\beta}+\frac{P(\beta c - \varphi_\beta^{*})}{\beta}$ we have
$$\frac{\|\varphi_\beta\|}{\beta}\leq \frac{2\,\|\varphi_\beta^{*}\|}{\beta} + \max{c} +\log(d)+\log(k+1)\leq K_2.$$

Using the estimative above and the fact that $X$ is a finite set we see that the Lipschitz constant of $\frac{\varphi_\beta}{\beta}$ is uniformly bounded, hence the Lipschitz constant of $c - \frac{\varphi_\beta}{\beta}$ is bounded.

It follows that for some subsequence, there exists the limit of $\frac{\varphi_\beta}{\beta}$. In the same way we get a control of the Lipschitz constants of the eigenfunctions $h_\beta(y)$ to $L_{\beta c -\varphi_\beta} $. Applying the Arzela-Ascoli Theorem we obtain the existence of the limit on $C^0$ norm of $\frac{h_\beta} \beta$ for some subsequence $\beta_n\to\infty$.

Suppose that for some sequence $\beta_n$ we have: $\frac{1}{\beta_n}\log(h_{\beta_n})(y) \to \tilde{V}(y)$, $\frac{\varphi_{\beta_n}(x)}{\beta_n} \to \tilde{m}(x)$ and $\pi_{\beta_n}\to \pi_\infty$. Then, $ \pi_\infty \in \Pi(\mu,\sigma)$, and
\begin{equation}\label{subaction}  \sup_{x}\sup_{a} [\,c(x,ay) +  \tilde{V}(ay) - \tilde{V}(y) - \tilde{m}(x)] = 0, \,\, \,\, \,\, \forall\, y.\end{equation}

On the other hand, from
\[ \int_{X\times \Omega} \beta c(x,y) -\varphi_\beta(x) \, d\pi_\beta + H(\pi_\beta)= \sup_{\pi\in \Pi(\mu,\sigma)}  \int_{X\times \Omega} \beta c(x,y) -\varphi_\beta(x)  \, d\pi + H(\pi),\]
we conclude that
\[ \int_{X\times \Omega} c(x,y) \, d\pi_\infty - \int_X \tilde{ m}(x) \, d\mu =\sup_{\pi\in \Pi(\mu,\sigma)}\int_{X\times \Omega} c(x,y)  \, d\pi - \int_X \tilde{m}(x) \, d\mu.\]
Therefore,
\[\int_{X\times \Omega} c(x,y) \, d\pi_\infty =\sup_{\pi\in \Pi(\mu,\sigma)}\int_{X\times \Omega} c(x,y)  \, d\pi,\]
which  means that $\pi_\infty$ is an optimal plan for $c$ over $\Pi(\mu,\sigma)$.

Now we prove the duality between the primal and dual problem \cite{LM}:
\begin{theorem} Let $\Phi$ be the set of functions $\alpha(x)$, such that, there exists a function $S(y)$ satisfying: $c(x,y) +S(y) -S(\sigma(y)) -\alpha(x) \leq 0$, for any $x\in X$ and $y\in\Omega$. Then,
\[\inf_{\alpha \in \Phi} \int_X \alpha(x) \, d\mu  = \sup_{\pi\in \Pi(\mu,\sigma)}\int_{X\times \Omega} c(x,y)  \, d\pi.\]
Moreover, suppose that for some sequence $\beta_n$ we have: $\frac{1}{\beta_n}\log(h_{\beta_n}) \to \tilde{V}$, $\frac{\varphi_{\beta_n}}{\beta_n} \to \tilde{m}$ and $\pi_{\beta_n}\to \pi_\infty$.
Then, the infimum on the left hand side is attained for $\alpha(x)=\tilde m(x)$ and $S(y)=\tilde V(y)$. The supremum  on the right hand side is attained in $\pi_\infty$. The function $\tilde V$ is a calibrated subaction.
\end{theorem}

\begin{proof}
 Given $\alpha$ and $S$ such that $c(x,y) +S(y) -S(\sigma(y)) -\alpha(x) \leq 0$, for any $x\in X,y\in\Omega$, we get  the inequality
\[\int_{X\times \Omega} c(x,y)\, d\pi \leq \int_X \alpha(x)\, d\mu ,\,\,\,\, \forall \pi \in \Pi(\mu,\sigma).\]
 Therefore,
\[\sup_{\pi\in\Pi(\mu,\sigma)} \int_{X\times \Omega} c(x,y)\, d\pi \leq \int_X \alpha(x)\, d\mu.\]
On the other hand, from the Variational Principle for $\Pi(\mu,\sigma)$ we get the equation
\[ \int_X \frac{\varphi_\beta(x)}{\beta} \, d\mu = \int_{X\times \Omega}  c(x,y) \, d\pi_\beta + \frac{1}{\beta}H(\pi_\beta).\]
Then, when $\beta_n\to+\infty$, we have that $\tilde m(x)\in \Phi$ and
\[\int_X \tilde m(x) \, d\mu = \int_{X\times \Omega} c(x,y)\, d\pi_{\infty}=\sup_{\pi\in \Pi(\mu,\sigma)}\int_{X\times \Omega} c(x,y)  \, d\pi.\]
\end{proof}

\section{Appendix} \label{App}

In the appendix we will give the proofs os some technical results.
\bigskip

\noindent
{\it Proof of Proposition \ref{Gi}.}
By the Schauder-Tychonov fixed point theorem we can find a plan $\pi_c \in P(X\times \Omega)$ such that $\hat L_c^*(\pi_c) = \pi_c$, for such normalized $c$. We note that $\pi_c\in\Pi(\cdot,\sigma)$ because
\begin{eqnarray*}
&&\int_{X\times \Omega} \psi\circ\sigma(y) \, d\pi_c\,=\int_{X\times \Omega} \psi\circ\sigma(y) \, d\hat L_c^*(\pi_c)\\
&&=\,\int_{X\times \Omega} \left(\sum_x\sum_{\sigma(w)=y} e^{c(x,w)} \psi\circ\sigma(w) \right) \, d\pi_c=\,\int_{X\times \Omega} \psi(y) \, d\pi_c \, .
\end{eqnarray*}

Now we show that the $y$-marginal of $\pi_c$ is $\nu_c$. Denote by $\tilde{\nu}_c$ the $y$-marginal of $\pi_c$.
 Then, for any fixed  $\psi\in C(\Omega)$,
 \begin{eqnarray*}
 \tilde{\nu}_c(\psi)&=&\pi_c(\psi)=\hat L_c^*(\pi_c)(\psi(y))=\int_{X\times \Omega} \left(\sum_x\sum_{\sigma(w)=y} e^{c(x,w)}\psi(w)\right) \, d\pi_c\\
 &=&\int_{ \Omega} \left(\sum_x\sum_{\sigma(w)=y} e^{c(x,w)}\psi(w)\right) \, d\tilde{\nu}_c =L_c^*(\tilde{\nu}_c)(\psi).
\end{eqnarray*}
Therefore, $\tilde{\nu}_c={\nu}_c$, because $\nu_c$ is the unique fixed point of $L_c^*$ .

The fixed point $\pi_c$ for $\hat L_c^*$  is unique because it satisfies, for any $u(z,y)$,
\begin{eqnarray*}
\pi_c(u)&=& \hat L_c^*(\pi_c)(u)\,=\,\int_{X\times \Omega} \Big(\sum_x\sum_{\sigma(w)=y} e^{c(x,w)} u(x,w) \Big) \, d\pi_c(z,y)\\
&=&\,\int_{ \Omega} \Big(\sum_x\sum_{\sigma(w)=y} e^{c(x,w)} u(x,w) \Big) \, d\nu_c(y).
\end{eqnarray*}

Finally, for a  fixed $(x_0,y_0)$ and an open set of the form $(x_0,A)$, where $A$ is a cylinder containing $y_0$ we have
\[ \pi_c((x_0,A))=\,\int_{\Omega} \Big(\sum_{\sigma(w)=y} e^{c(x_0,w)} \textbf{I}_A(w) \Big) \, d\nu_c(y)>0,\]
because $\nu_c$ is positive on cylinders and  $e^{c(x_0,w)}$ is bounded below.
In this way we show that the support of $\pi_c$ is the full set $X\times \Omega$.
%
%

$\hfill \Box$


\medskip

Now we will explaining how the Jacobian of a plan $\pi$ is defined.
We will adapt the reasoning of \cite{PY} to our setting.

Let $\cal{B}$ be the Borel sigma-algebra over $X\times\Omega$. Moreover, let $\sigma^{-1}(\cal{B})$  be the sigma algebra generated by cylinders of the form $[\cdot,\cdot \,y_1...y_n], n=1,2...$ where
\[[\cdot,\cdot \,y_1...y_n]=\{(x,(w_0,w_1,...))\in X\times \Omega:w_1=y_1,...,w_n=y_n\}.\]
Remember that for each $(x,a)\in X\times \{1,...,d\}$,  $[x,a]=\{(z,(w_0,w_1,...)):z=x,w_0=a\}$.

Given a plan $\pi$ with $y$-marginal $\nu$ we define for each $(x,a)$ the measure $\pi^{x,a}$ over $\sigma^{-1}(\cal{B})$  by the rule $\pi^{x,a}(A) =\pi([x,a]\cap A)$. Clearly $\pi^{x,a} \ll\pi$ then, from the Radon-Nikodym Theorem, there exists a function
\[E(\mathbb{I}_{[x,a]}\,|\,\sigma^{-1}(\mathcal{B}))
:=\frac{d \pi^{x,a}}{d \pi} \in L^{1}(X\times\Omega,\sigma^{-1}(\cal{B}),\pi),\]
which is the conditional expectation of $\mathbb{I}_{[x,a]}$ given
$\sigma^{-1}(\mathcal{B})$.

In the same way, for each  $n$, we consider $\mathcal{B}_n$ which is the smallest sigma-algebra containing the cylinders of the form $[\cdot,\cdot \,y_1...y_n]$. For each $(x,a)$ let $\pi_n^{x,a}$ over $\mathcal{B}_n$ be defined by $\pi_n^{x,a}(A) =\pi([x,a]\cap A)$. Applying again the Radon-Nikodym Theorem we get a function
\[E(\mathbb{I}_{[x,a]}\,|\,\mathcal{B}_n):=\frac{d\pi_n^{x,a}}{d\pi} \in L^{1}(X\times\Omega,\mathcal{B}_n,\pi).\]
Note that
\[\int_{X\times\Omega} \mathbb{I}_{[\cdot,\cdot \,y_1...y_n]} \,d\pi_n^{x,a} = \int_{X\times\Omega} \mathbb{I}_{[x,ay_1...y_n]} \,d\pi = \pi([x,ay_1...y_n]),\]
and
\[\int_{X\times\Omega} \mathbb{I}_{[\cdot,\cdot \,y_1...y_n]} \,d\pi_n^{x,a} = \int_{X\times\Omega} \mathbb{I}_{[\cdot,\cdot \,y_1...y_n]}E(I_{[x,a]}\,|\,\mathcal{B}_n) \,d\pi.\]
Then, using the fact that
$E(\mathbb{I}_{[x,a]}\,|\,\mathcal{B}_n)$ is constant on the set $[\cdot,\cdot \,y_1...y_n]$,
 we get
\[ E(\mathbb{I}_{[x,a]}\,|\,\mathcal{B}_n)(x_0,(y_0,y_1,...,y_n,...)) = \frac{\pi([x,ay_1...y_n])}{\pi([\cdot,\cdot \,y_1...y_n])} = \frac{\pi([x,ay_1...y_n])}{\nu([y_1...y_n])}.\]

Let $$J_{\pi}^{n}:=\sum_{x,a}\mathbb{I}_{[x,a]}E(\mathbb{I}_{[x,a]}\,|\,\mathcal{B}_n),$$
then, for $(x_0,(y_0,y_1,...))\in X\in \Omega$,
$$J_{\pi}^{n}(x_0,y) = \frac{\pi([x_0,y_0y_1...y_n])}{\nu([y_1...y_n])}.  $$
From the increasing martingale theorem, when $n\to\infty$,
\[E(\mathbb{I}_{[x,a]}\,|\,\mathcal{B}_n) \to E(\mathbb{I}_{[x,a]}\,|\,\sigma^{-1}(\mathcal{B}))\]
in $L^{1}(X\times\Omega,\sigma^{-1}(\mathcal{B}),\pi)$ and in a.e. $\pi$.
Then, by summing over $(x,a)$ we get  a function $J_\pi$ well defined $\pi$ a.e., such that,
\[J_{\pi}^{n} \to J_{\pi}\]
in $L^{1}(X\times\Omega, \mathcal{B},\pi) $ and a.e. $\pi$.
Note that:
\[J_\pi = \sum_{x,a}\mathbb{I}_{[x,a]}E(\mathbb{I}_{[x,a]}\,|\,\sigma^{-1}(\mathcal{B}))\]
in $L^{1}(X\times\Omega, \sigma^{-1}(\mathcal{B}),\pi).$

\bigskip

Following the terminology of \cite{PP} and \cite{PY} we mention that the information function is defined by
$$I:=-\log(J_\pi) = -\sum_{x,a}\mathbb{I}_{[x,a]}\log(E(\mathbb{I}_{[x,a]}\,|\,\sigma^{-1}(\mathcal{B}))).$$

In this case the entropy of $\pi$ is
\[H(\pi):=\int_{X\times\Omega} -\log(J_\pi)\, d\pi = - \int_{X\times\Omega} \sum_{x,a}\mathbb{I}_{[x,a]}\log(E(\mathbb{I}_{[x,a]}\,|\,\sigma^{-1}(\mathcal{B})))\, d\pi.\]
The number $H(\pi)$ is finite.

This is the end of the basic considerations about the concepts of Jacobian and entropy of a plan.

\bigskip

\noindent
{\it Proof of Theorem \ref{gibbs-equilibrium}.}
Let $\lambda_c$ be the main eigenvalue and $h_c$ the positive eigenfunction of $L_c$, given by Proposition \ref{RPF}, then $\overline{c}(x,y):= c(x,y) +\log(h_c)(y) - \log(h_c\circ\sigma)(y) - \log(\lambda_c)$ is the normalized cost associated to $c$. As $h_c$ depends only on $y$, for any $\pi \in \Pi(\cdot,\sigma)$ we have that
$\displaystyle -\int_{X\times \Omega} \bar c\, d\pi =-\int_{X\times \Omega} c\, d\pi +\log(\lambda_c).$ Hence, by the definition of  entropy, we get
\begin{eqnarray*}
P(c) &=& \sup_{\pi \in \Pi(\cdot,\sigma)} \left(\int_{X\times \Omega} c\, d\pi + H(\pi)\right)\\
&\leq&  \sup_{\pi \in \Pi(\cdot,\sigma)} \left(\int_{X\times \Omega} c\, d\pi -\int_{X\times \Omega} \bar c\, d\pi \right) = \log(\lambda_c).
\end{eqnarray*}

Now we show the other inequality: let  $\pi_{\bar{c}}$ be the Gibbs plan  to  $\bar c$.
 Then, by Lemma \ref{L2},
 $\displaystyle H(\pi_{\bar{c}})= - \int_{X\times \Omega} \bar c \,d\pi_{\bar{c}}=-\int_{X\times \Omega} c   \,d\pi_{\bar{c}}+ \log(\lambda_c).$

Therefore,
\[P(c) = \sup_{\pi \in \Pi(\cdot,\sigma)} \left(\int_{X\times \Omega} c\, d\pi + H(\pi)\right) \geq   \left(\int_{X\times \Omega} c \, d\pi_{\bar{c}} + H(\pi_{\bar{c}})\right)=\log(\lambda_c).\]

In order to prove that the equilibrium plan is unique let us suppose that $c$ is normalized. Then $P(c)=0$ and for all $\pi\in\Pi(\cdot,\sigma)$ we have
$$\int_{X\times \Omega} c\,d\pi-\int_{X\times \Omega} \log (J_{\pi })\,d\pi=\int_{X\times \Omega} c\,d\pi+H(\pi)\leq 0 ,  $$
with equality, if and only if, $c=\log(J_\pi)$,  by Lemma \ref{L1a}.
Suppose now $\pi$ is such that $\int_{X\times \Omega} c\,d\pi+H(\pi)=0$.
Using Lemma \ref{L01}, for every $w\in C(X,\Omega)$,
\[\int_{X\times \Omega} \sum_{x}\sum_{a} J_\pi(x,ay)w(x,ay)\,d\pi = \int_{X\times \Omega} w\, d\pi,\] hence, as $J_{\pi}=e^c$,
\[\int_{X\times \Omega} \sum_{x}\sum_{a} e^{c(x,ay)}w(x,ay)\,d\pi =\int_{X\times \Omega} \hat L_c (w)d\pi= \int_{X\times \Omega} w\, d\pi.\]
This shows that $\hat{L}^*_{{c}} (\pi)=\pi.$
Finally, from the uniqueness of the Gibbs plan given by Proposition \ref{Gi}, we get that $\pi=\pi_{{c}}.$

$\hfill \Box$

Now we will prove some other results that we used before.

\medskip

\medskip



\noindent
{\it Proof of Lemma \ref{L01}.}
We need to prove that, for every $w\in C(X,\Omega)$,
$\pi \in \Pi(\cdot,\sigma)$,
\[\int_{X\times \Omega} \sum_{x}\sum_{a} J_\pi(x,ay)w(x,ay)\,d\pi = \int_{X\times \Omega} w\, d\pi.\]

First we show that if $w$  is constant in the cylinders of the form $[x,y_0...y_n]$, then
\[\int_{X\times \Omega} \sum_{x}\sum_{a} J_\pi^{n}(x,ay)w(x,ay)\,d\pi = \int_{X\times \Omega} w\, d\pi.\]

Consider a function
$w_n= \mathbb{I}_{[i,j_0j_1...j_n]}$. Then,
\begin{eqnarray*}
&&\int_{X\times \Omega} \sum_{x}\sum_{a} J_\pi^{n}(x,ay)w_n(x,ay)\,d\pi \\
&&=\int_{X\times \Omega} \sum_{x}\sum_{a} \frac{\pi([x,ay_0...y_{n-1}])}{\nu([y_0...y_{n-1}])} \mathbb{I}_{[i,j_0j_1...j_n]}(x,ay)\,d\nu(y)\\
&&=\pi([i,j_0j_1...j_n]) = \int_{X\times \Omega} w_n\, d\pi.
\end{eqnarray*}
From linearity arguments we conclude the first part of the Lemma.

In order to prove the second part of the Lemma we take a function  $w_l= \mathbb{I}_{[i,j_0j_1...j_l]}$. Then, using the first part of the Lemma, we obtain
\begin{eqnarray*}\int_{X\times \Omega} \sum_{x}\sum_{a} J_\pi(x,ay)w_l(x,ay)\,d\pi&= &\lim_{n\to\infty}\int_{X\times \Omega} \sum_{x}\sum_{a} J_\pi^{n}(x,ay)w_l(x,ay)\,d\pi\\
&= &\int_{X\times \Omega} w_l\, d\pi,
\end{eqnarray*}
where we use that, if $n\geq l$, $w_l$ is also constant in the cylinder of the form $[x,y_0...y_n]$.

From linearity arguments and using the fact that the functions which are constant in cylinders of length $l=1,2,3,...$ are dense in $C(X,\Omega)$ we conclude the proof.

$\hfill \Box$

\medskip
\medskip

\medskip

The Proof of Lemma \ref{L1a} will require the following:


\begin{lemma} \label{Jacap}
If $b$ is a normalized potential which is constant on cylinders of the form $[x,y_0...y_n]$, then
\[ - \int_{X\times \Omega} \log(J_\pi^{n}) \, d\pi \leq - \int_{X\times \Omega} b \, d\pi.\]
Furthermore, there exists a family of normalized potentials $b_\epsilon$ such that
\[- \int_{X\times \Omega} \log(J_\pi^{n}) \, d\pi =\lim_{\epsilon\to 0} - \int_{X\times \Omega} b_\epsilon \, d\pi.\]
\end{lemma}

\begin{proof}
Let us fix a normalized potential $b$ constant on cylinders of the form $[x,y_0...y_n]$.
The functions $u=\frac{e^{b}}{J_\pi^{n}}$ and $\log(J_\pi^{n})$ are well defined in  supp($\pi$).
Using Jensen inequality, we have
\begin{eqnarray*} 0 &=& \int_{X\times \Omega} \log\left(\sum_x\sum_a e^{b(x,ay)}\right)\, d\pi=\int_{X\times \Omega} \log \left(\sum_x\sum_a J_\pi^{n}(x,ay)u(x,ay)\right)\, d\pi\\
&\geq& \int_{X\times \Omega} \left(\sum_x\sum_a J_\pi^{n}(x,ay)\log(u(x,ay))\right)\, d\pi=\int_{X\times \Omega} \log(u(z,y))\, d\pi\\
&=& \int_{X\times \Omega} b(z,y) - \log(J_\pi^{n}(z,y))\, d\pi.
\end{eqnarray*}
This shows the first part of the lemma.

In order to show the second part, we consider for each cylinder $[y_1...y_n]$, such that, $\nu([y_1...y_n])>0$, the sets
\newline
$ A=A_{[y_1...y_n]}:=\{(x,a)\in X\times \{1,...,d\}\,:\,\pi([x,ay_1...y_n])=0\}$ and
\newline
$B=B_{[y_1...y_n]}:=\{(x,a)\in X\times \{1,...,d\}\,:\,\pi([x,ay_1...y_n])>0\}$.

Fixed $\epsilon>0$ sufficiently small, consider the potential $b_\epsilon$ defined by
\[ b_\epsilon([x,ay_1...y_n])=\left\{\begin{array}{ll} \log((\#B)\,\epsilon), & if\, (x,a)\in A\\ \log(J_\pi^{n} ([x,ay_1...y_n]) -(\#A)\,\epsilon), & if\,(x,a)\in  B \end{array}\right.\]
If $\nu([y_1,...,y_n])=0$, we define $b_\epsilon([x,ay_1...y_n])=-\log((\#X)\,d)$, for all $x\in X, a\in \{1,...,d\}\ $.
By construction we see that $b_\epsilon$ is a normalized potential. Indeed, take $z=(z_1,z_2,...)\in\Omega$, there are two cases:
\newline
 If $\nu([z_1...z_n])=0$, then $$\sum_x\sum_a e^{b_\epsilon(x,az)} =\sum_x\sum_a e^{b_\epsilon([x,az_1...z_n])} =\sum_x\sum_a e^{-\log((\#X)\,d)}=1.$$
 If $\nu([z_1...z_n])>0$, then 
 \begin{eqnarray*}&&\sum_x\sum_a e^{b_\epsilon(x,az)}=\sum_A e^{\log((\#B)\,\epsilon)} +\sum_B e^{\log(J_\pi^{n}([ x,az_1...z_n]) -(\#A)\,\epsilon)}\\
 &&= (\#A)(\#B)\,\epsilon + \left(\sum_B J_\pi^{n} ([x,az_1...z_n])\right)-(\#B)(\#A)\,\epsilon\\
 &&= \sum_B J_\pi^{n}([ x,az_1...z_n])=1.
\end{eqnarray*}
 When $\epsilon\to 0$,
 \begin{eqnarray*}
 \lim_{\epsilon\to 0} \int_{X\times \Omega} b_\epsilon \,d\pi &=& \lim_{\epsilon\to 0}\sum_{\nu([y_1...y_n])>0}\sum_{B_{[y_1...y_n]}} b_\epsilon([x,y_0y_1...y_n])\pi([x,y_0y_1...y_n])\\
 &=& \int_{X\times \Omega} \log(J_\pi^{n})\,d\pi.
\end{eqnarray*}

\end{proof}


\medskip
{\it Proof of Lemma \ref{L1a}.}
First, we need to prove that, if $b$ is a normalized potential and $\pi\in \Pi(\cdot,\sigma)$, then
\[0\leq  - \int_{X\times \Omega} \log(J_\pi) \, d\pi \leq - \int_{X\times \Omega} b \, d\pi,\]
with equality,  if and only if, $b=\log(J_\pi)$.



We know that $\log(J_\pi^{n})$ converges to $\log(J_\pi)$ a.e. ($\pi$). Following the Lemma 8.11 and Theorem 8.12 in \cite{PY} we conclude that $\log(J_\pi^{n})$ converges to $\log(J_\pi)$  in $L^{1}$ norm.

We claim  that, for all $b$ normalized
\[ - \int_{X\times \Omega} \log(J_\pi) \, d\pi \leq - \int_{X\times \Omega} b \, d\pi.\]
  Indeed, note that functions $u=\frac{e^{b}}{J_\pi}$ and $\log(J_\pi)$ are well defined a.e $\pi$ and that $\log$ is a strictly concave function, then by Jensen inequality we have
  \begin{eqnarray*}
  0 &=& \int_{X\times \Omega} \log\left(\sum_x\sum_a e^{b(x,ay)}\right)\, d\pi=\int_{X\times \Omega} \log \left(\sum_x\sum_a J_\pi(x,ay)u(x,ay)\right)\, d\pi\\
  &\geq&  \int_{X\times \Omega} \left(\sum_x\sum_a J_\pi(x,ay)\log(u(x,ay))\right)\, d\pi= \int_{X\times \Omega} \log(u(z,y))\, d\pi\\
  &=&\int_{X\times \Omega} b(z,y) - \log(J_\pi(z,y))\, d\pi,
\end{eqnarray*}
proving the claim.

In the case $\int_{X\times \Omega} b(z,y) - \log(J_\pi(z,y))\, d\pi=0$ we get that for $\pi$ a.e. $y$, the Jensen's inequality will be an equality \footnote{   $\,\displaystyle 0=\log \bigg(\sum_x\sum_a J_\pi(x,ay)u(x,ay)\bigg)\geq \sum_x\sum_a J_\pi(x,ay)\log(u(x,ay))$ is an equality iff $u(x,ay)=k(y)$ for all $x\in X ,a\in\{1,...,d\}$, hence $0=\log ( k(y)) $ this implies $k(y)=1$.}. Then, for $\pi$ a.e. $y$ we have $u(x,ay)$ is constant equal to $1$
. That is, for almost all $y$ we have that for any $x$ and $a$ the equality $\log J_\pi(x,ay)= b(x,ay)$  hold. Using Lemma \ref{L01} it follows that $\pi$ is the Gibbs plan for $b$, because of  the uniqueness assertion of  Proposition \ref{Gi}. From  Lemma \ref{L2} we get that  $J_\pi=e^b$, hence $\log J_\pi=b$.

Now, the final claim of Lemma \ref{L1a},
\[- \int_{X\times \Omega} \log(J_\pi) \, d\pi =\inf_{b\,normalized} - \int_{X\times \Omega} b \, d\pi.\]
is a consequence of the second part from Lemma \ref{Jacap}.


$\hfill \Box$

\medskip

\noindent
{\it Proof of Proposition \ref{entropia-prod}.}
 We need to prove that, given $\pi \in \Pi(\cdot,\sigma)$,  if the $x$-marginal of $\pi$ is a probability measure $\mu$ and the $y$-marginal of $\pi$ is an invariant measure $\nu$, then
\[H(\pi) \leq h(\mu) + h(\nu) .\]
Moreover, if $\pi = \mu\times\nu$, then $H(\pi) = h(\mu)+h(\nu)$.

We remember that the Kolmogorov entropy of $\nu$ satisfies
\[h(\nu) = \inf\left\{ - \int_{ \Omega} g(y)\, d\nu(y)\,: \, g \,\text{Lipschitz and}\, \sum_{\sigma(w)=y}e^{g(w)}=1\right\}.\]
Given $\epsilon>0$, let $g$ be a Lipschitz function satisfying $\sum_{\sigma(w)=y}e^{g(w)}=1$ and $- \int_{ \Omega} g(y)\, d\nu(y)< h(\nu)+\epsilon$. The potential $c(x,y)=g(y)+\log(\mu(x))$ is normalized and therefore
\begin{eqnarray*}
H(\pi) &\leq& -\int_{X\times \Omega}  g(y)+\log(\mu(x)) \, d\pi = -\int_{ \Omega}  g(y)\, d\nu(y) -\int_X  \log(\mu(x))\,d\mu(x)\\
&\leq& h(\nu)+h(\mu)+\epsilon.
\end{eqnarray*}
Taking $\epsilon\to 0$, we conclude the first part of the proof.

Now we suppose that $\pi=\mu\times\nu$. Then
\begin{eqnarray*}
H(\pi)&=&\lim_{n\to\infty}-\int_{X\times \Omega} \log\left(\frac{\pi([x,y_0...y_n])}{\nu([y_1...y_n])}\right)\,d\pi(x,y)\\
&=&\lim_{n\to\infty}-\int_{X\times \Omega} \log(\mu(x)) + \log\left(\frac{\nu([y_0...y_n])}{\nu([y_1...y_n])}\right)\,d\pi(x,y)\\
&=&-\sum_x \log(\mu(x))\mu(x) -\lim_{n\to\infty} \int_{ \Omega}\log\left(\frac{\nu([y_0...y_n])}{\nu([y_1...y_n])}\right)\,d\nu(y)\\
&=&h(\mu)+h(\nu).
\end{eqnarray*}
$\hfill \Box$



Now we will show the last proof that was missing.
\medskip

\noindent
{\it Proof of Proposition \ref{pressure-properties}. }
 Itens (a) and (b) are straightforward.

To prove item (c), i.e. the convexity of the pressure map, we suppose that $c= \lambda c_1+ (1-\lambda)c_2$. Given $\epsilon>0$, there exists $\pi_{\epsilon}$, such that,
\begin{eqnarray*}
  P(c)- \epsilon & \leq & \int_{X\times \Omega} c d\pi_{\epsilon} + H(\pi_{\epsilon}) \\
   & =    &\lambda \left( \int_{X\times \Omega} c_1 d\pi_{\epsilon}+H(\pi_{\epsilon}) \right)+(1-\lambda) \left(\int_{X\times \Omega} c_2 d\pi_{\epsilon}+H(\pi_{\epsilon}) \right)  \\
   &\leq & \lambda P(c_1) + (1-\lambda) P(c_2).
\end{eqnarray*}

In order to prove item (d), for a given $\epsilon>0$, there exists $\pi_{\epsilon}$, such that,

\begin{eqnarray*}
P(c_1)-P(c_2) & \leq  & \int_{X\times \Omega} c_1 d\pi_{\epsilon}+ H(\pi_{\epsilon})+\epsilon - P(c_2) \\
& \leq & \int_{X\times \Omega} c_1 d\pi_{\epsilon}+ H(\pi_{\epsilon})+\epsilon - \left( \int_{X\times \Omega} c_2 d\pi_{\epsilon}+ H(\pi_{\epsilon})\right) \\
& \leq  & \int_{X\times \Omega} |c_1-c_2| d \pi_{\epsilon} + \epsilon \\
&  \leq &
\|c_1-c_2\| + \epsilon.
\end{eqnarray*}

$\hfill \Box$

\end{document}